\documentclass{article}

\usepackage[T1]{fontenc}
\usepackage[utf8]{inputenc}

\usepackage{multirow} 
\usepackage{amsfonts,mathtools,amssymb}
\usepackage{graphicx,wrapfig,tikz,animate,subcaption}
\usepackage{lmodern}
\usepackage{fancyhdr,graphicx,wrapfig,amsmath,amsthm,amssymb, mathtools, scrextend, titlesec, enumitem,color}
\usepackage{array}
\usepackage{hyperref}
\usepackage{caption}
\usepackage{comment}
\usepackage{cancel}
\usepackage{nicefrac}
\newcommand{\vvvert}{\ensuremath|\!|\!|\xspace}
\usepackage[normalem]{ulem}

\newcommand{\yesdelete}[1]{}

\usepackage[boxed]{algorithm2e} 
\usetikzlibrary{arrows}

\newcommand{\est}{\text{\,e.s.t.}\xspace}

\definecolor{verylightgray}{gray}{0.93}
\usepackage[color=verylightgray]{todonotes}

\usepackage{wrapfig,tikz}
\usetikzlibrary{shapes.misc}

\newtheorem{thm}{Theorem}

\newtheorem{lemma}[thm]{Lemma}
\newtheorem{propo}[thm]{Proposition}

\newtheorem{rem}[thm]{Remark}
\usepackage{xspace}

\newcommand{\N}{\ensuremath{\mathbb{N}}\xspace}
\newcommand{\R}{\ensuremath{\mathbb{R}}\xspace}

\newcommand{\Q}{\ensuremath{\mathcal{Q}}\xspace}
\newcommand{\T}{\ensuremath{\mathcal{T}}\xspace}
\newcommand{\aaa}{\ensuremath{\mathbf{a}}\xspace}
\newcommand{\xxx}{\ensuremath{\mathbf{x}}\xspace}

\DeclareMathOperator{\Div}{div}
\DeclareMathOperator{\supp}{supp}
\newcommand{\Span}{\text{span}}
\newcommand{\tfa}{\text{ for all }}
\newcommand{\qtfa}{\quad\text{for all }}

\newcommand{\Hoi}[1][\Omega]{\ensuremath{H_0^1(#1)}\xspace}

\newcommand{\bT}[1][{}]{\ensuremath{\psi_T^{#1}}}

\newcommand{\bS}[1][{}]{\psi_S^{#1}}
\newcommand{\patchS}[1][{}]{\omega_S^{#1}}
\newcommand{\nz}{\mathbb{N}}
\newcommand{\rz}{\mathbb{R}}
\newcommand{\ueps}{u^\epsilon}
\newcommand{\uas}{u^{as}}
\newcommand{\intO}{\int_\Omega}
\newcommand{\uBas}{\ensuremath{u_B^{as}}\xspace}
\newcommand{\psilambda}{\ensuremath{\psi_{\lambda}}}
\newcommand{\psilambdaas}{\ensuremath{\psi_{\lambda}^{as}}}

\newcommand{\summLA}[1][{}]{\ensuremath{\sum_{\lambda\in\Lambda_{#1}}\alpha_\lambda}}
\newcommand{\summL}[1][{}]{\ensuremath{\sum_{\lambda\in\Lambda_{#1}}}}
\newcommand{\Gammaleft}{\ensuremath{\Gamma_\lambda^{\text{left}}}\xspace}
\newcommand{\Gammaright}{\ensuremath{\Gamma_\lambda^{\text{right}}}\xspace}
\newcommand{\nsep}{\ensuremath{n\sqrt{\epsilon}}}

\newcommand{\oab}{\ensuremath{1+\alpha+\beta}}
\newcommand{\otab}{\ensuremath{1+2\alpha+\beta}}
\newcommand{\hatT}{\ensuremath{{\hat{T}}}}
\newcommand{\hatb}{\ensuremath{\hat{b}}}
\newcommand{\hatx}{\ensuremath{{\hat{x}}}}
\newcommand{\haty}{\ensuremath{{\hat{y}}}}
\newcommand{\hatD}{\ensuremath{\hat{\Delta}}}
\newcommand{\hatN}{\ensuremath{\hat{\nabla}}}
\newcommand{\hatnsep}{\ensuremath{n\sqrt{\hat{\epsilon}}}}
\newcommand{\hatnep}{\ensuremath{n\hat{\epsilon}}}

\definecolor{b1}{RGB}{34,139,34}   
\definecolor{b2}{RGB}{255,140,0}   
\definecolor{bT}{RGB}{30,60,150}   

\title{Error estimates for the patch bubble method for convection-dominated channel flow problem}

\author{Eberhard B\"ansch\footnote{Applied Mathematics III,
Friedrich-Alexander-University Erlangen-Nuremberg,
Cauerstr. 11, 91058 Erlangen, Germany
\href{mailto:baensch@math.fau.de}{baensch@math.fau.de}}
\and 
Pedro Morin\footnote{Universidad Nacional del Litoral and CONICET, 
Departamento de Matem\'atica, 
Facultad de Ingenier\'{\i}a Qu\'{\i}mica, Santiago del Estero 2829, S3000AOM Santa Fe, Argentina. 
\href{mailto:pmorin@fiq.unl.edu.ar}{pmorin@fiq.unl.edu.ar}}
\and Itatí Zocola\footnote{Universidad Nacional del Litoral and CONICET, 
Departamento de Matem\'atica, 
Facultad de Ingenier\'{\i}a Qu\'{\i}mica, Santiago del Estero 2829, S3000AOM Santa Fe, Argentina. 
\href{mailto:izocola@fiq.unl.edu.ar}{izocola@fiq.unl.edu.ar}}
}
\date{}

\begin{document}

	\maketitle


\begin{abstract}
We present error estimates for the BMZ (Bubble Mesh Zoom) residual-free bubble method applied to a convection–diffusion equation in the convection-dominated regime.
The method incorporates both element bubbles and residual-free bubbles supported on patches of two adjacent elements.

We focus on the case of a parallel flow in a square domain and derive error estimates in an energy norm that remain valid as diffusion becomes small.
The theoretical findings are corroborated by numerical experiments, which also exhibit a competitive performance of the method.
\end{abstract}

\noindent\textbf{Keywords:} 
Residual free bubbles; convection-diffusion; finite elements

\noindent\textbf{AMS Subject Classification:}
65M60 (Primary) 65N30, 65N12 (Secondary)

\tableofcontents


\section{Introduction}

The goal of this article is to present a theoretical analysis of the residual-free bubble method introduced in~\cite{BMZ} for advection-dominated problems, where the diffusion coefficient $0<\epsilon \ll 1$. In that work, we proposed augmenting the discrete space not only with element-based residual-free bubbles (RFB), but also with bubbles defined on two-element patches. This enrichment increases the flexibility of the approximation space and enables a more accurate resolution of boundary layers. The numerical experiments reported in~\cite{BMZ} demonstrate competitive performance and excellent stabilization properties on a variety of stationary and instationary test problems. However, no error analysis was provided in that article. In the present article we prove error estimates for the case of channel flow.

For a detailed discussion of both element-based and patch-based residual-free bubble methods, we refer the reader to~\cite{BMZ} and the references therein.

Theoretical results and error estimates for the classical RFB method can be found, for example, in 
\cite{BFHR1997, Russo1997, Franca1998, Brezzi1999}. The idea of employing patch bubbles was first introduced and analyzed in \cite{Cangiani2005long,Cangiani2005short}, where the authors referred to the approach as the enhanced residual-free bubble method (RFBe). In \cite{Cangiani2005long}, error estimates are derived for a flow directed along a diagonal relative to the square domain. These estimates, however, degenerate when the advective field aligns with one of the coordinate axes.

The available error estimates reveal several structural features of these formulations. For the classical RFB method \cite[Thm. 4.2]{Brezzi1999}, the bounds typically involve derivatives of the exact solution, which may become unbounded as $\epsilon$ becomes small; thus, the estimate deteriorates precisely in the convection-dominated regime. For the RFBe method, the derived bound \cite[Thm. 2]{Cangiani2005long} has the form
 $\max\{(\epsilon/h)^{1/2},\epsilon^{1/4}\} + h$. The error estimates we obtain in this article are of the form $(\epsilon/h)^{1/4}$, but for a norm that is a bit stronger than the one considered in \cite[Thm. 2]{Cangiani2005long}; see Theorem \ref{Thm:main}.
It is worth observing that none of these estimates guarantees convergence purely as the mesh parameter $h$ tends to $0$, but they are all valid in the convection-dominated regime $\epsilon \ll h$. In the diffusion regime, when $h < \epsilon$, the methods do not require stabilization and the usual estimates for the pure Galerkin approximation are valid, providing optimal convergence rates.

The rest of this paper is organized as follows. In Section
\ref{Sec:problem} we introduce the problem, its discretization, our proposal of stabilization using bubbles,
as well as our main result, Theorem \ref{Thm:main}.
Section~\ref{Sec:ErrorEstimate} is devoted to the proof of the main result.
To this end, we rely on the asymptotic expansion introduced and analyzed in~\cite{GieEtAl:13}, and we devote Subsection~\ref{Sec:interpolant} to the construction of a suitable interpolant for this expansion.
In Section~\ref{Sec:numerics} we present numerical experiments that demonstrate how our method compares computationally with existing approaches.
Conclusions are drawn in Section~\ref{Sec:conclusions}.

Finally, two appendices provide technical auxiliary results, which are necessary to prove the main theorem. 
They were left to the end in order to ease the reading of Section~\ref{Sec:ErrorEstimate}.


\section{The problem and its discretization}\label{Sec:problem}
\subsection{The problem}
We consider the advection-dominated advection diffusion problem defined, for $0< \epsilon \le \epsilon_0 \ll 1$
 and $\aaa\in \R^2, |\aaa|=1$ as follows:
\begin{equation}\label{Eq:problem}
    - \epsilon \Delta u^\epsilon + \aaa \cdot \nabla  u^\epsilon =f \quad \text{in $\Omega$,}\qquad
        u^\epsilon = 0 \quad \text{on $\partial\Omega$}
\end{equation}
with $\Omega\subset \R^2$ open and bounded, and $f \in L^2(\Omega)$.
In order to define a weak solution we let the bilinear form $a(\cdot,\cdot): \Hoi \times \Hoi \rightarrow \rz$ be defined as

\begin{equation*}
a(u,v) := \epsilon\int_\Omega \nabla u \cdot \nabla(e^{-\aaa\cdot \xxx} v) \, d\xxx + \int_\Omega \aaa \cdot \nabla u\, v \,e^{-\aaa\cdot \xxx} \,d\xxx,
\end{equation*}
and the linear form $F$ on $L^2(\Omega)$:
\begin{equation*}
F(v) := \int_\Omega f\, v\, e^{-\aaa\cdot \xxx} \, d\xxx.
\end{equation*}

We consider the following weak formulation of~\eqref{Eq:problem}:
\begin{equation}\label{eq:weak}
\text{Find } \ueps \in \Hoi : \qquad a(\ueps,v) = F(v), \quad \forall v\in\Hoi.
\end{equation}

Due to the exponential term inside the bilinear form, we have the following coercivity result, which immediately implies the well-posedness of~\eqref{eq:weak}.

\begin{lemma}[Coercivity]
\label{Lemma:Coercivity}
There exists a constant $\alpha > 0$, independent of $\epsilon$, and solely depending on $\Omega$ such that
$$
 a(v,v) \ge \alpha \| v\|_\epsilon^2\qtfa v \in \Hoi,
$$
where
\begin{equation}
\label{eq:epsnorm}
  \| v \|_\epsilon^2 := \epsilon \| \nabla v\|^2 + \|v\|^2,
\end{equation}
with $\| \cdot \|$ the usual $L^2(\Omega)$ norm.
\end{lemma}

Whenever we say that a constant is independent of $\epsilon$, we understand that it may depend on $\epsilon_0$ and that we are considering $0<\epsilon\le\epsilon_0$.
We will usually denote $x \lesssim y$ to denote $x\le C y$, with a constant $C>0$ which may depend on $\epsilon_0$ and $\Omega$, but is otherwise independent of the other involved quantities. 

\begin{proof}
Notice that for any $v\in \Hoi$,
$$
  a(v,v) = \epsilon\int_\Omega e^{-\aaa\cdot \xxx} |\nabla v|^2\, d\xxx + \epsilon \intO (-\aaa \cdot \nabla v) \, v e^{-\aaa\cdot \xxx}\, d\xxx
     + \intO \aaa \cdot \nabla v \, v e^{-\aaa\cdot \xxx}\,d\xxx.
$$
Integration by parts yields
\begin{equation*}
\begin{aligned}
  \intO (\aaa \cdot \nabla v) \, v e^{-\aaa\cdot \xxx}\, d\xxx 
   &= -\intO v \Div(\aaa \, v\, e^{-\aaa\cdot \xxx}) \, d\xxx 
  \\
&= -\intO \aaa \cdot \nabla v \, v e^{-\aaa\cdot\xxx} \, d\xxx
+ \intO e^{-\aaa \cdot \xxx} |\aaa|^2 |v|^2 \, d\xxx,
\end{aligned}
\end{equation*}
whence, using that $|\aaa|=1$,
$  \intO (\aaa \cdot \nabla v) \, v e^{-\aaa\cdot \xxx}\, d\xxx 
   = \frac12
\intO e^{-\aaa \cdot \xxx} |v|^2 \, d\xxx$.
Thus,
\begin{equation}
  a(v,v) = \epsilon\int_\Omega e^{-\aaa\cdot \xxx} |\nabla v|^2 \,d\xxx
  + \frac{1-\epsilon}{2} \intO e^{-\aaa\cdot \xxx} |v|^2\,d \xxx
  \gtrsim \| v\|_\epsilon^2.
  \tag*{\qedhere}
\end{equation}
\end{proof}

\subsection{Discretization: Patch bubble space.}

We assume that $\Omega$ is subdivided into squares with side length $h>0$, and denote by $\T$ the partition of $\Omega$ consisting of all the squares $T$.
Let the finite element space $V_L$ be defined as
\[
V_L := \{ v \in H^1(\Omega) : v_{|T} \in \Q^1, \forall T\in \T, 
\ v(x) = 0, \text{ for all boundary nodes $x$}  \},
\]
with $\Q^1$ the space of bilinear functions of the form 
$a_{00} + a_{10} x + a_{01} y + a_{11} x y$.

Our proposal consists in enriching the discrete space $V_L$ by adding:
\begin{itemize}
\item for each element $T$, the space $B_T := \Span\{\bT[i], i=1,\dots, 4\}$, where $\bT[i]$, $i=1,\dots,4$, are the \emph{element bubbles} defined as the solutions to:
\[
-\epsilon \Delta \bT[i] + \aaa \cdot \nabla \bT[i] 
= \varphi_T^i 
\quad\text{in $T$}, 
\qquad 
\bT[i] = 0 
\quad\text{on $\partial T$},
\] 
with $\varphi_T^i$, $i=1,\dots, 4$ the nodal basis functions corresponding to the bilinear elements;
\item and also for each edge $S$ of the skeleton $\Sigma_h$ of the partition, the space $B_S := \Span\{\bS\}$, with the residual-free bubble $\bS$ defined as follows:
\[
-\epsilon \Delta \bS + \aaa \cdot \nabla \bS = l 
\quad\text{in $\patchS$}, 
\qquad 
\bS = 0 \quad\text{on $\partial\patchS$},
\]
where the patch $\patchS$ is the union of the two elements of the partition sharing the side $S$ (see Figure \ref{F:PatchDomains})
and $l$ is a particular right hand side, see Subsection \ref{Sec:mainresult}. Here, $\Sigma_h$ denotes the set of interior edges of the partition. 
\end{itemize}

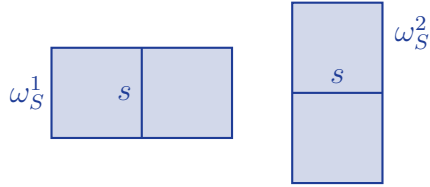
\begin{figure}[h]
    \centering
    \begin{tikzpicture}[scale=0.4]
    
        \fill[bT, opacity=0.2] (0,1.5) rectangle (6,4.5);
        \draw[bT, thick] (0,1.5) rectangle (6,4.5);

         \fill[bT, opacity=0.2] (8,0) rectangle (11,6);
         \draw[bT, thick] (8,0) rectangle (11,6);
    
        \draw[bT, thick] (3,1.5) -- (3,4.5) node[midway, left] {\large $s$};
        \draw[bT, thick] (8,3) -- (11,3) node[midway, above] {\large $s$};

    \node at (-0.8,3) {\Large\textcolor{bT}{ \large$\patchS[1]$}};
    \node at (12,5) {\Large\textcolor{bT}{ \large$\patchS[2]$}};
    
    \end{tikzpicture}
    \caption{\small \label{F:PatchDomains} Example of patch domains $\patchS[1]$ and $\patchS[2]$, corresponding to vertical and horizontal edges of a square partition.}
\end{figure}

Finally, we set
\[
    V_h := V_L \oplus V_B,
    \quad\text{with}\quad
    V_B := \bigoplus_{T\in\T_h} B_T\oplus\bigoplus_{S\in\Sigma_h}B_S.
\]
The discrete problem associated to \eqref{Eq:problem} now reads:
\begin{equation}\label{eq:discrete problem}
\text{Find } u_h:=u_L + u_B  \in V_h:\qquad    a(u_h,v_h) = F(v_h) \qtfa v_h \in V_h.
\end{equation}
We refer to this method as \emph{bubble mesh zoom} (BMZ) in the sequel.
\begin{figure}[h!tbp]
    \centering
    \includegraphics[width=.49\textwidth]{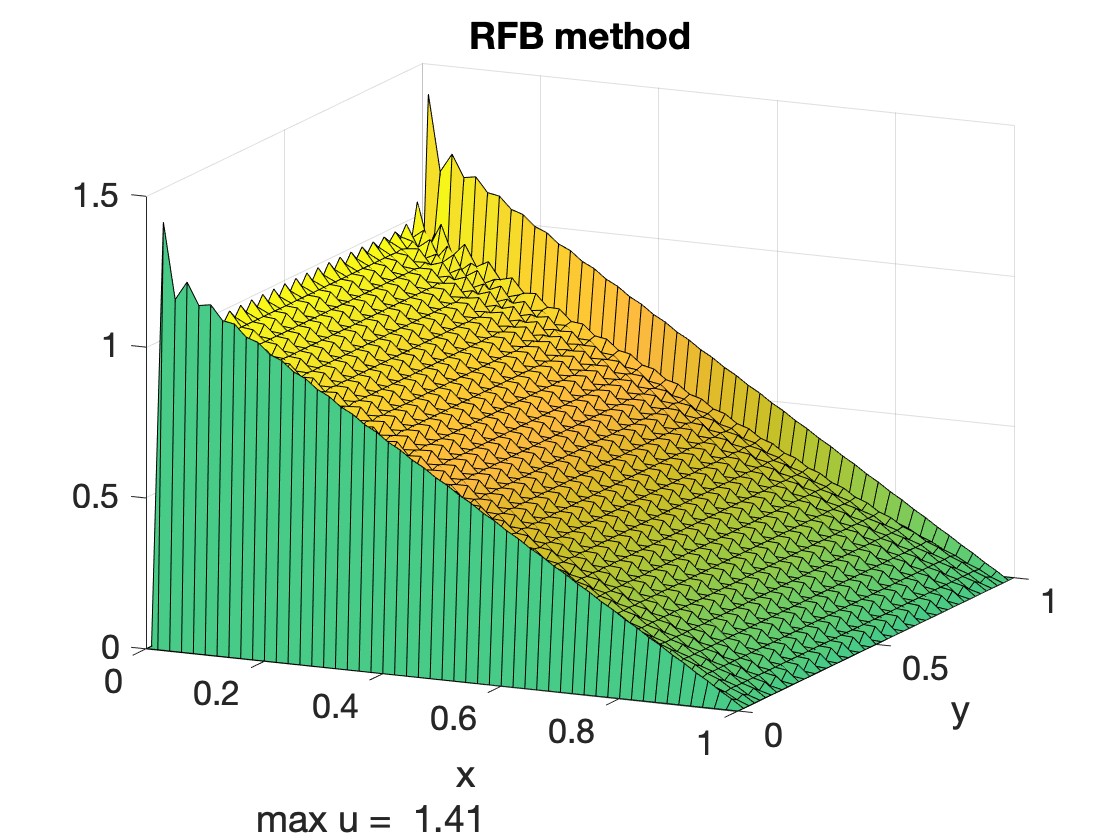}
    \includegraphics[width=.49\textwidth]{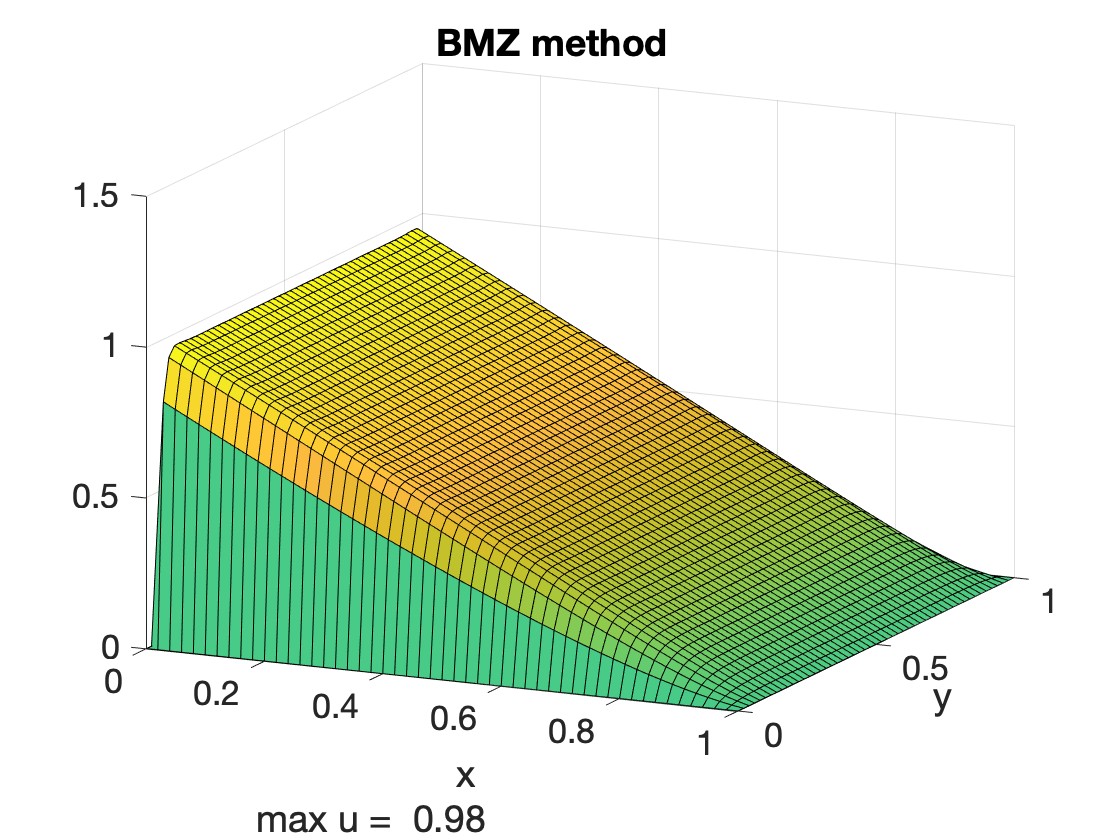}
    \caption{\small \label{F:DiscreteSolutions}
    Solutions obtained with the residual-free bubbles method (left) and BMZ method (right), corresponding to $f=1$, $\aaa = (-1,0)$ and $\epsilon=10^{-6}$. The partition consists of $50\times 50$ square elements.
    The proposed BMZ method eliminates boundary layer oscillations and corner spikes observed in the RFB solution, resulting in a
smoother global behavior, without any observable smearing effect. Additionally, it improves the maximum value, reducing it
from 1.4099 (RFB) to  0.9841 (BMZ).}
\end{figure}

\begin{rem}
\rm
The classical RFB method considers $V_B =\bigoplus_{T\in\T_h} B_T$, i.e., only the \emph{element bubbles} without \emph{edge bubbles}.
An advantage of this method is the simplicity with which the final discrete solution can be computed. After testing separately with functions in $V_L$ and $V_B$ a static condensation eliminates the bubble unknowns, yielding a linear system solely for $v_L$. Importantly, this requires no explicit evaluation of the bubble functions themselves: only elementwise integrals involving the residual appear, and these are straightforward to approximate numerically.
Because of this elimination procedure, the resulting method 
coincides exactly with the well-known SUPG formulation \cite{BFHR1997}, with the RFB framework providing a precise expression for the stabilization parameters.
\end{rem}
\begin{rem}[Another method using patch-bubbles]
\rm
The first papers that proposed the use of patch bubbles were those by Cangiani and S\"uli~\cite{Cangiani2005long,Cangiani2005short} with the so called RFBe method. They considered bubbles which satisfy the homogeneous equation in each element of the patch, and are equal to the analytical 1D solution (corresponding to the projected wind) on the interelement edge. Besides some impressive numerical experiments, they obtain some error estimates for the case of a somewhat diagonal wind, by resorting to the asymptotic approximation presented by Schieweck~\cite{thesisSchieweck}. Those error estimates blow up when the angle between the wind direction and part of the boundary vanishes, in principle, by a limitation in Schieweck's formulas.
Besides, we predict that their method will not be as good in such a case because those bubbles will not be able to capture precisely the (parabolic) boundary layer.
Instead, we propose to use bubble functions that solve the same advection-diffusion problem on the patches, thereby capturing the essential behavior of the solution near the boundaries, regardless of the wind direction.
\end{rem}

\begin{rem}[Computing the bubbles]\rm
The reader might have noticed that the residual-free bubbles are solutions to a problem as difficult as the original one; albeit with an important difference: the domain is now of size $h$.
Cangiani and S\"uli use Shishkin meshes in order to compute these bubbles.
Our approach, instead, leverages \emph{recursion}.
The algorithm, which from now on we refer to as the BMZ (Bubble-Mesh-Zoom) method, is based on the idea of zooming in locally to solve bubble sub-problems over each mesh element or patch.
The entries of the system matrix are computed within a function called $\texttt{element\_contributions}$, in which we generate finite element partitions on the \emph{patches}, and use the same method with patch bubbles to compute the bubbles on the coarser mesh, if $h>\epsilon/|\aaa|$; otherwise, a usual Galerkin method with bilinear finite elements is used. 
The recursion will thus reach the limit of $h < \epsilon/|\aaa|$ in $O(\log \frac{|\mathbf{a}|h}{\epsilon})$ recursive calls.
For further details on the implementation and performance we refer the reader to~\cite{BMZ}.
We finally mention that, as in~\cite{Cangiani2005long}, we perform the numerical analysis of this article assuming that the bubbles are computed exactly; this allows us to understand the essence of the method, without digging into extremely technical details.
\end{rem}

In order to quickly show the advantage of using patch bubbles, we present in Figure~\ref{F:DiscreteSolutions} the graphs obtained for the residual-free bubble stabilization and the proposed method, corresponding to~\eqref{Eq:problem} with $f=1$, $\aaa = (-1,0)$ and $\epsilon=10^{-6}$, on meshes of size $50\times50$, i.e. $h=0.02$. 
The solution obtained using the newly proposed approach effectively eliminates the oscillations near the boundary layer and the sharp spike in the corner observed for the RFB solution. Furthermore, the solution produced by the BMZ method exhibits a satisfactory global behavior, reinforcing its reliability for this problem. Notably, the maximum value achieved with the RFB method is $1.4099$, whereas the proposed method yields a significantly improved maximum value of $0.9841\cong1-h$.


We end this section with the following observation regarding the optimality of the approximation $u_h$, which will be useful for our analysis, and is an immediate consequence of the coercivity stated in Lemma~\ref{Lemma:Coercivity}:
\begin{equation}\label{inf-sup}
\sup_{0\not= v_h\in V_h} \frac{a(w_h,v_h)}{\|v_h\|_\epsilon} \geq \alpha \|w_h\|_\epsilon
\quad\tfa w_h\in V_h.
\end{equation}
This inf-sup condition implies the well-posedness of the discrete problem~\eqref{eq:discrete problem} and the following optimal approximation result, with a slightly modified energy norm, which coincides with the one used by Schieweck~\cite{thesisSchieweck}.

\begin{lemma}[Cea's Lemma]
\label{Lemma:Cea}
Let $\ueps$ denote the exact solution and $u_h$ the discrete one. Then
$$
\vvvert \ueps - u_h\vvvert  \lesssim \inf_{z_h\in V_h} \vvvert  \ueps - z_h\vvvert ,
$$
where the norm $\vvvert \cdot\vvvert $ is defined in $H^1(\Omega)$ by
\begin{equation}
\label{eq:||| norm}
\vvvert  v\vvvert   := \|v\|_\epsilon + 
\sup_{0\not= z_h\in V_h} \frac{a(v,z_h)}{\|z_h\|_\epsilon}.
\end{equation}
\end{lemma}

\begin{proof}
Let $z_h\in V_h$. Due to~\eqref{inf-sup}, there exists $v_h \in V_h$ such that
\begin{equation*}
\|z_h - u_h\|_\epsilon \lesssim \frac{a(z_h-u_h,v_h)}{\|v_h\|_\epsilon}.
\end{equation*}
Due to Galerkin orthogonality,
\[
\|z_h - u_h\|_\epsilon \lesssim
\frac{a(z_h-\ueps,v_h)}{\|v_h\|_\epsilon} 
\leq \sup_{0\not= w_h\in V_h} \frac{a(z_h-\ueps,w_h)}{\|w_h\|_\epsilon} \leq \vvvert z_h -\ueps\vvvert .
\]

Using the triangular inequality we arrive at
$$
\|\ueps - u_h\|_\epsilon \leq \|\ueps - z_h\|_\epsilon + \|z_h-u_h\|_\epsilon
    \lesssim \vvvert \ueps  - z_h\vvvert .
$$
Moreover, again by Galerkin orthogonality one gets
$$
\|\ueps - u_h\|_\epsilon = \vvvert  \ueps - u_h\vvvert ,
$$
which leads to the desired assertion.
\end{proof}

\subsection{Main result}\label{Sec:mainresult}

In this section we state the main theoretical result of this paper,
which is an error estimate in the particular case of $\Omega = (0,1)^2$, $\aaa = (-1,0)$ and $f \equiv 1$.
It is worth noticing that, in this setting, the solution exhibits three boundary layers: one (elliptic) of width $\epsilon$ at $x=0$, and two (parabolic) of width $\sqrt{\epsilon}$ at $y=0$ and $y=1$.

Estimates on general domains and with general right-hand sides as precise and robust as the ones we obtain here seem to be unreachable for the moment.
We focus on a particular case which allows us to infer the potential capabilities of our proposed method.

We will perform the error estimates in the particular case that only one bubble per element and one per each side touching the boundary are used:
\begin{itemize}
\item for each element $T$, we consider the space $B_T := \Span\{\bT
\}$, where $\bT$ is the \emph{element bubble} defined as the solution to:
\[
-\epsilon \Delta \bT + \aaa \cdot \nabla \bT
= l^T(x,y) 
\quad\text{in $T$}, 
\qquad 
\bT = 0 
\quad\text{on $\partial T$};
\] 
where
\begin{equation}\label{def:l_b-l_tT}
l^T(x,y) = \begin{cases}
    l_b(y):=1-y/h, \quad&\text{if $T$ touches $\{y=0\}$},
    \\
    l_t(y):=1-(1-y)/h,  \quad&\text{if $T$ touches $\{y=1\}$},
    \\
    1, \quad&\text{otherwise}.
\end{cases}
\end{equation}

\item and also for each edge $S\in\Sigma_h^+$, we consider the space $B_S := \Span\{\bS\}$, with the residual-free bubble $\bS$ defined as follows:
\[
-\epsilon \Delta \bS + \aaa \cdot \nabla \bS = l^S(x,y), 
\quad\text{in $\patchS$}, 
\qquad 
\bS = 0 \quad\text{on $\partial\patchS$},
\]
where:
\begin{equation}\label{def:l_b-l_tS}
l^S(x,y) = \begin{cases}
    l_b(y), \quad&\text{if $S$ is vertical and touches $\{y=0\}$},
    \\
    l_t(y),  \quad&\text{if $S$ is vertical and touches $\{y=1\}$},
    \\
    1, \quad&\text{otherwise}.
\end{cases}
\end{equation}
Here, $\Sigma_h^+$ denotes the set of interior edges of the partition that touch the outflow boundary $\{x=0\}\cup\{y=0\}\cup \{y=1\}$. 
\end{itemize}
\begin{thm}[Main result]\label{Thm:main}
    Let $\ueps$ be the exact solution of~\eqref{eq:weak} and let $u_h$ be the discrete solution to~\eqref{eq:discrete problem}, both corresponding to $\Omega = (0,1)^2$, $\aaa = (-1,0)$ and $f \equiv 1$.
    Then, there exists a constant $C$, independent of $\epsilon$ and $h$ and $\hat{\epsilon}_0 >0$ such that
    \[ \vvvert \ueps - u_h\vvvert \le C (\epsilon/h)^{1/4}, \]
if $0<\epsilon/h \leq \hat{\epsilon}_0$.
\end{thm}

In order to prove this result we resort to the asymptotic expansion defined and analyzed in~\cite{GieEtAl:13}, which is of the form
\begin{equation*}
\uas := u^0 + \text{correctors},
\end{equation*}
where $u^0$ denotes the $0$-th order approximation to $\ueps$, given by the solution to
\begin{equation*}
    -  \partial_x u^0 =1 \quad\text{in $\Omega$,}\qquad
        u^0 = 0 \quad\text{at } x=1,
\end{equation*}  
i.e., $u^0(x,y) = 1-x$.
The term \emph{correctors} refers to additional functions introduced to compensate for the incorrect boundary values of $u^0$. Higher-order asymptotic approximations are discussed in  \cite{GieEtAl:13} 
, but for our purposes, the zeroth-order approximation (i.e., $n = 0$ in \cite{GieEtAl:13}) is sufficient. In this case, it is shown in Lemma~\ref{Temam:ErrorBound} that
\begin{equation*}
  \| \ueps - \uas \|_\epsilon \lesssim \epsilon.
\end{equation*}

Using this approximation $\uas$, we let $\ueps_I$ denote a suitable interpolant of $\uas$, 
and then exploit Cea's Lemma~\ref{Lemma:Cea} to estimate
\begin{equation}\label{eq:Cea}
\vvvert \ueps - u_h\vvvert  \lesssim \vvvert \ueps - \ueps_I \vvvert  \le  \vvvert  \ueps - \uas\vvvert  + \vvvert  \uas - \ueps_I\vvvert .
\end{equation}

\medskip

In what follows, we will say that a function $f$ is $\est$ in $G\subseteq\Omega$ if, for all $m\in\nz$, there are constants $c_1,c_2,\gamma > 0$, independent of $\epsilon$, such that
$$
\| f\|_{H^m(G)} \leq c_1 \exp(\frac{-c_2}{\epsilon^\gamma}).
$$
Also, we will use $C$ to denote a constant independent of $\epsilon$ and the discretization parameter $h$, which may change from one line to the next one and $A \lesssim B$ will denote $A \le C\, B$.

\section{Error estimate}\label{Sec:ErrorEstimate}

This section is dedicated to estimating the second term in~\eqref{eq:Cea}, i.e., the error $\vvvert  \uas - \ueps_I\vvvert $ between the asymptotic expansion $\uas$ and a suitable interpolant $\ueps_I$ belonging to the discrete space $V_h$.
As we have already mentioned,
the first term on the right-hand side of~\eqref{eq:Cea} is already bounded by $\epsilon$, due to the results from \cite{GieEtAl:13} (see Lemma~\ref{Temam:ErrorBound}). We have included in Appendix \ref{Appendix:Temam} the specific estimates from~\cite{GieEtAl:13} that we need in our context.

In order to bound the second term, we first propose a candidate for the interpolant $\ueps_I$, as the sum of a function $u_L \in V_L$ and a function $u_B \in V_B$.
Secondly, to bound this error, we
introduce $u_B^{as}$, an asymptotic approximation of $u_B$. 
Taking into account that $u_B$ will be a combination of exact solutions on smaller domains, 
the approximation function $u_B^{as}$ will be constructed by resorting again to the formulas from~\cite{GieEtAl:13}.

\subsection{Asymptotic approximation}

As we mentioned in the introduction, we consider the asymptotic expansion $\uas$ defined and analyzed in~\cite{GieEtAl:13}, corresponding to $n=0$, which is of the form
\begin{equation*}
\uas := u^0 + \underbrace{\varphi+\xi+\theta+\zeta}_\text{correctors},
\end{equation*}
where $u^0(x,y) = 1-x$.

Several correctors are given, which imitate the behavior of the exact solution in the boundary layers:
$$
\begin{aligned}
\text{correctors} = {}&\phantom{+}\; \varphi \text{ (accounting for the parabolic boundary layer at } y=0, y=1) \\
  &+ \xi \text{ (accounting for the elliptic boundary layer)} \\
   &+ \theta \text{ (accounting for the ordinary boundary layer at } x=0) \\
 &+ \zeta \text{ (accounting for the corner layers)}
 .
\end{aligned}
$$

Using the results from Appendix \ref{Appendix:Temam}, specifically Lemma \ref{Temam:Lemma2}, Lemma \ref{Temam:eq35} and Lemma \ref{Temam:Lemma6}, 
 one can bound the correctors by
\begin{equation}\label{Eq:correctors}
\|\varphi\|_\epsilon, \|\xi\|_\epsilon, \|\zeta\|_\epsilon \lesssim \epsilon^{1/4}.
\end{equation}

In the original paper~\cite{GieEtAl:13} there is also a complementary corrector $\eta$ \cite[Eq. 53]{GieEtAl:13}, which vanishes in our case, because we consider $n=0$.

However, the crucial corrector is $\theta$, which accounts for the strong elliptic boundary layer
at $x=0$ and cannot be bounded in the same way as the others. This corrector must be compensated by the corresponding corrector for the interpolant (see Section~\ref{Sec:elliptic}).

\subsection{Some auxiliary estimates and abstract error estimates }


Below, we present two lemmas, based on results from \cite{GieEtAl:13}, that provide bounds of the error and the correctors in the $\vvvert \cdot\vvvert $-norm. 

\begin{lemma}\label{Lemma:4}
Let $\ueps$ be the exact solution of~\eqref{eq:weak} and $\uas$ be its asymptotic expansion. Then
$$
 \vvvert  \ueps - \uas\vvvert  \lesssim \epsilon + \frac{1}{\sqrt{\epsilon}} \|\ueps - \uas\|
               \lesssim \epsilon + \sqrt{\epsilon} \lesssim \sqrt{\epsilon}.
$$
\end{lemma}

\begin{proof}
From Lemma~\ref{Temam:ErrorBound} we have 
$
\|\ueps - \uas\|_\epsilon \lesssim \epsilon.
$
For $v_h \in V_h$ we estimate
\begin{equation*}
  \begin{aligned}
   a(\ueps - \uas,v_h) 
    &= \epsilon\intO \nabla (\ueps - \uas)\nabla (v_h) e^x dx\,dy 
          + \epsilon\intO \partial_x(\ueps - \uas)v_h e^x dx\,dy \\ 
     &\qquad  -\intO \partial_x(\ueps - \uas)v_h e^x dx\,dy \\
      &\lesssim \|\ueps - \uas\|_\epsilon \|v_h\|_\epsilon  -(1-\epsilon)\intO \partial_x(\ueps - \uas) v_h e^x dx\,dy.
  \end{aligned}
\end{equation*}
Integrating by parts, one gets
\begin{equation*}
  \begin{aligned}
       -(1-\epsilon)\intO \partial_x(\ueps - \uas)v_h e^x dx\,dy 
      &= (1-\epsilon)\intO (\ueps - \uas)v_h e^x dx\,dy \\
      & \quad+ (1-\epsilon)\intO (\ueps - \uas) \partial_x v_h e^x dx\,dy \\
   & \lesssim\|\ueps -\uas\|_\epsilon\|v_h\|_\epsilon + \|\ueps-\uas\|\;\|\partial_x v_h\|,
  \end{aligned}
\end{equation*}
so that 
\[ | a(\ueps - \uas,v_h) |
      \lesssim \|\ueps - \uas\|_\epsilon \|v_h\|_\epsilon  +\|\ueps-\uas\|\;\|\partial_x v_h\|.
      \]
      
Since $\|\partial_x v_h\| \le 1/\sqrt{\epsilon} \|v_h\|_\epsilon $ and $\|\ueps - \uas\| \lesssim \|\ueps - \uas\|_\epsilon \lesssim \epsilon$ we have
\begin{equation*}
   |a(\ueps - \uas,v_h) |
 \lesssim \epsilon \|v_h\|_\epsilon +  \sqrt{\epsilon}\|v_h\|_\epsilon,
\end{equation*}
and the assertion follows.
\end{proof}

We also need to bound the correctors $\varphi$, $\xi$, $\zeta$ for $\ueps$ in the $\vvvert \cdot \vvvert $ norm.

\begin{lemma}\label{Lemma:5}
For the correctors $\varphi, \xi, \zeta$ we have the following bounds:

$$
\vvvert  \varphi\vvvert , \; \vvvert \xi\vvvert ,\; \vvvert \zeta\vvvert  \lesssim \epsilon^{1/4}.
$$
\end{lemma}

\begin{proof}
In \eqref{Eq:correctors} a bound in the $\|\cdot\|_\epsilon$-norm was already stated. Similar as in the proof of
Lemma \ref{Lemma:4} for $0\not= v_h\in V_h$ we estimate
$$
\begin{aligned}
\frac{a(\varphi,v_h)}{\|v_h\|_\epsilon} &\lesssim \|\varphi\|_\epsilon
        + (1-\epsilon)\frac{|\intO \partial_x \varphi \,v_h e^xdx\,dy|}{\|v_h\|_\epsilon} \\
   &\lesssim \|\varphi\|_\epsilon  + \frac{\|\partial_x \varphi\|\,\|v_h\|}{\|v_h\|_\epsilon} 
   \lesssim \|\varphi\|_\epsilon  + \|\partial_x \varphi\| \lesssim \epsilon^{1/4},
\end{aligned}
$$
by Lemma \ref{Temam:Lemma2}.

Now, integrating by parts, we estimate $\zeta$:
$$
\begin{aligned}
\frac{a(\zeta,v_h)}{\|v_h\|_\epsilon} 
&\leq 
\|\zeta\|_\epsilon + \frac{|\intO x \partial_x \zeta \frac{v_h e^x}{x}dx\,dy|}{\|v_h\|_\epsilon} 
\lesssim \|\zeta\|_\epsilon + \frac{\epsilon^{3/4} \|\nabla(v_he^x)\|}{\|v_h\|_\epsilon},
\end{aligned}
$$
where the last estimate follows from Lemma~\ref{Temam:Lemma6} and Hardy's inequality.
Finally,
$$
\begin{aligned}
\frac{a(\zeta,v_h)}{\|v_h\|_\epsilon} 
        &\lesssim \|\zeta\|_\epsilon + \frac{\epsilon^{3/4} \|\nabla(v_he^x)\|}{\|v_h\|_\epsilon}
        = \|\zeta\|_\epsilon + \frac{\epsilon^{1/4} \epsilon^{1/2}\|\nabla(v_he^x)\|}{\|v_h\|_\epsilon} 
              \lesssim \epsilon^{1/4}.
\end{aligned}
$$
Noting that $\xi$ fulfills $-\epsilon \Delta \xi - \partial_x \xi =0$ in $\Omega$ (see \cite[Eq. 34]{GieEtAl:13}), we have that
$
a(\xi,v_h) = 0
$
because $v_h$ vanishes on $\partial\Omega$.
\end{proof}

\subsection{Accurate interpolant} \label{Sec:interpolant} 

For $i,j=1,\dots N$ define the element $T_{i,j}$ with left bottom vertex $(x_i,y_j)$, $x_i= 1- ih$ and
$ y_j = (j-1)h$. The edge $S_{i+1/2,j}$ is the one shared by element $T_{i,j}$ and $T_{i+1,j}$ and
$S_{i,j+1/2}$ the one shared by $T_{i,j}$ and $T_{i,j+1}$ (see Figure \ref{Fig:T_Labels}).

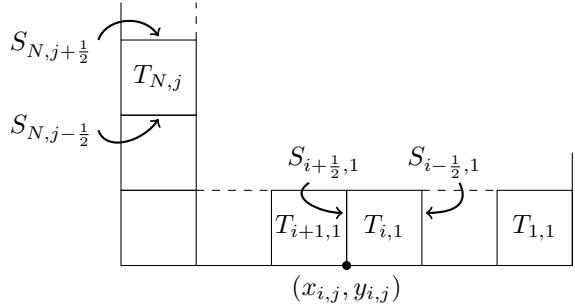
\begin{figure}[ht]
\begin{center}
\begin{tikzpicture}[scale=1]
    \def\h{1}
    \def\H{1}

    \foreach \j in {0,...,2} {
        \draw (0,\j*\H) rectangle ++(\h,\H);
    }

    \draw (0,3*\H) -- (0*\h,3.5*\H);
    \draw[dashed] (\h,3*\H) -- (\h,3.5*\H);


    \node at (0.5*\h,2.5*\H) {$T_{N,j}$};

    \draw[->, thick] (-0.3*\h,3*\H) .. controls +(0.1,0.3) and +(-0.2,0.3) .. (0.5*\h,3.05*\H);
    \draw[->, thick] (-0.3*\h,1.8*\H) .. controls +(0.1,-0.3) and +(-0.2,-0.3) .. (0.5*\h,1.95*\H);

    \node[right] at (-1.6*\h,2.9*\H) {$S_{N,j+\frac{1}{2}}$};
    \node[right] at (-1.6*\h,1.8*\H) {$S_{N,j-\frac{1}{2}}$};

    \foreach \i in {2,3,5} {
        \draw (\i*\h,0) rectangle ++(\h,\H);
    }

    \draw (1*\h,0) -- (2*\h,0);
    \draw[dashed] (1*\h,\H) -- (2*\h,\H);
    
    \draw (4*\h,0) -- (5*\h,0);
    \draw[dashed] (4*\h,\H) -- (5*\h,\H);

    \draw (6*\h,0) -- (6*\h,1.5*\H);

    \node at (2.5*\h,0.5*\H) {$T_{i+1,1}$};
    \node at (3.5*\h,0.5*\H) {$T_{i,1}$};
    \node at (5.5*\h,0.5*\H) {$T_{1,1}$};
    
    \filldraw (3*\h,0) circle (1.5pt);
    \node[below] at (3*\h,0) {$(x_{i,j}, y_{i,j})$};

    \draw[->, thick] (2.5*\h,1.1) .. controls +(-0.3,-0.2) and +(-0.3,0) .. (2.95*\h,0.7*\H);
    \draw[->, thick] (4.5*\h,1.1) .. controls +(0.1*\h,-0.2) and +(0.5*\h,0) .. (4.05*\h,0.7*\H);

    \node[below] at (2.7*\h,1.7*\H) {$S_{i+\frac{1}{2},1}$};
    \node[below] at (4.3*\h,1.7*\H) {$S_{i-\frac{1}{2},1}$};

\end{tikzpicture}
\end{center}
\vspace{-0.7 cm}
  \caption{\small\label{Fig:T_Labels}
  Labeling of elements on the bottom parabolic and left elliptic boundaries.}

\end{figure}


    
    

    
    
    





The construction of the functions
$u_L$ and $u_B$ defining the interpolant $\ueps_I$ starts with the introduction of several index sets. First, we define $\Lambda$, as the set of indices corresponding to the bubbles whose supports intersect the boundary layers:
\begin{equation}\label{def:Lambda+u_B}
\Lambda := \Lambda_B \cup \Lambda_T \cup \Lambda_L,
\qquad\text{and}\qquad
u_B = \sum_{\lambda \in \Lambda} \alpha_\lambda \psi_\lambda ,
\end{equation}
with
\begin{equation*}
\begin{aligned}
    \Lambda_B &= \{ (i,1),(i+1/2,1): i=1, \ldots N-1\},\\
    \Lambda_T &= \{ (i,N),(i+1/2,N): i=1, \ldots N-1\},\\
    \Lambda_L &= \{ (N,j): j=1,\ldots,N \}
    \cup \{ (N,j+1/2): j=1,\ldots,N-1 \}.
\end{aligned}
\end{equation*}
This definition ensures that we include both element bubbles and patch bubbles associated with elements adjacent to the boundaries where layers are expected to form (i.e., on $y=0$, $y=1$ and $x=0$).
Given a bubble index $\lambda \in \Lambda$, we define $\Omega_\lambda$ as the support of $\psi_\lambda$, and the coefficients $\alpha_\lambda$ to be determined in the following sections.

Given $\lambda\in\Lambda$, for the bubble $\psilambda$ we can define correctors $\theta_{\lambda},\varphi_{\lambda},\xi_{\lambda},\zeta_{\lambda}$ in $\Omega_\lambda$ following the definition in~\cite{GieEtAl:13}, as previously done. Such that
\[\psilambdaas =\psi_\lambda^0 + \theta_{\lambda}+\varphi_{\lambda}+\xi_{\lambda}+\zeta_{\lambda}. \]  Summing up over all elements in the boundary layers, with the corresponding coefficients, we obtain global correctors for $u_B$:
\begin{equation}\label{eq:defBubbCorr}
    \theta_B := \summLA \theta_\lambda,
    \quad
    \varphi_B := \summLA \varphi_\lambda,
    \quad
    \xi_B := \summLA \xi_\lambda,
    \quad
    \zeta_B := \summLA \zeta_\lambda,
\end{equation}
which instead allow us to express \uBas as
\begin{equation}\label{def:uB_ass}
\uBas = u_B^0 + \underbrace{\theta_B+ \varphi_B +\xi_B + \zeta_B}_{\text{correctors for }u_B},
\qquad\text{with } u_B^0  = \summLA \psi_\lambda^0.
\end{equation}

We will now define the interpolant as $\ueps_I:=u_L + u_B$ with $u_L \in V_L$ and bound $\vvvert  \uas - \ueps_I\vvvert $ as follows:
\begin{align*}
\vvvert  \uas - \ueps_I\vvvert &=  \vvvert \uas - (u_L + u_B) +\uBas - \uBas\vvvert  \\
       & \leq \vvvert u^0 - (u_L + u_B^0)\vvvert   + \vvvert \text{correctors for }\uas\vvvert  + \vvvert \text{correctors for }u_B\vvvert  \\
       &\qquad + \vvvert u_B^{as} - u_B\vvvert .
\end{align*}
In Proposition~\ref{prop:bilinear0} we prove that $u^0 - (u_L + u_B^0) = 0$ and devote Section~\ref{Sec:correctors} to bound the other terms; summarizing the results in Theorem~\ref{Thm:final error estimate}.




\subsubsection{Interpolant in the interior ($u_L$)}

Define ${\mathaccent 23 \Omega}_h := \bigcup \{T_{i,j} \;|\; i=1, \dots N-1,\; j=2,N-1\}$, the domain $\Omega$
without the elements touching the parabolic and elliptic boundaries.

In ${\mathaccent 23 \Omega}_h$ we define
$$
\ueps_I := u_L = \text{piecewise bilinear interpolant of } u^0.
$$
Since $u^0 = 1-x$, in ${\mathaccent 23 \Omega}_h$ the interpolation error fulfills 
$\ueps - \ueps_I = u^0 - \ueps_I + \est = \est$ in ${\mathaccent 23 \Omega}_h$.
The bilinear interpolant $u_L$ is extended to $\bar{\Omega}$ by prescribing zero values on the boundary $\partial\Omega$. 

\subsubsection{Interpolant in the parabolic boundary layers}

In the elements at the top and bottom (parabolic boundary) it holds
$$
\begin{aligned}
u^0 - u_L &= l_b(y)(1-x) = (1- \frac{y}{h})(1-x),\quad\text{if } 0\leq y \leq h \text{ and } \\
u^0 - u_L &= l_t(y)(1-x) = (1 + \frac{1}{h}(y-1))(1-x),\quad\text{if } 1-h \leq y \leq 1.
\end{aligned}
$$
We now determine coefficients $\alpha_\lambda$ for $\lambda\in\Lambda_B\cup\Lambda_T$ such that
$
u^0 - u_L- u_B^0
$
is small, where $u_B$ satisfies
\begin{equation}\label{Eq:psiI}
u_B = \summLA[B] \psi_\lambda + \summLA[T] \psi_\lambda, 
\qquad\text{in } (h,1-h)\times\big((0,h)\cup (1-h,h)\big),
\end{equation}
i.e., in the elements touching the parabolic boundary.
For the residual-free bubbles $\psi_\lambda$,
 in $T_{i,1}$ and $T_{i,N}$, for $i=1,2,\dots,N-1$, we have
\begin{equation}\label{Eq:psi}
    \left\{
\begin{aligned}
     \psi^0_{i,1}        &= l_b(y)(x_{i-1}-x), \\
        \psi^0_{i-1/2,1} &= l_b(y)(x_{i-2}-x),\\
        \psi^0_{i+1/2,1} &= l_b(y)(x_{i-1}-x).\\
\end{aligned}
    \right.
    \quad\text{and}\quad
    \left\{
\begin{aligned}
        \psi^0_{i,N}     &= l_t(y)(x_{i-1}-x), \\
        \psi^0_{i-1/2,N} &= l_t(y)(x_{i-2}-x),\\
        \psi^0_{i+1/2,N} &= l_t(y)(x_{i-1}-x).\\
\end{aligned}
    \right.
\end{equation} 
With $l_b$ and $l_t$ defined by~\eqref{def:l_b-l_tT} and~\eqref{def:l_b-l_tS}.
In order to make $u^0 - u_L - u_B^0$ small, we make the following choice of coefficients, for $i=2,\dots N$:
\begin{equation}\label{Eq:coefficients}
    \left\{
\begin{aligned}
        \alpha_{1,1} &:=0, \\
        \alpha_{i-1/2,1} &:= i-1, \\
        \alpha_{i,1}  &:= -2\alpha_{i-1/2,1} = 2 -2 i.
\end{aligned}
    \right.\quad
    \left\{
\begin{aligned}
        \alpha_{1,N} &:=0, \\
        \alpha_{i-1/2,N} &:= i-1, \\
        \alpha_{i,N}  &:= -2\alpha_{i-1/2,N} = 2 -2 i.
\end{aligned}
    \right.
\end{equation}

Figure~\ref{fig:lateralView} presents a schematic lateral view of the bubbles on the parabolic boundary, with the coefficients given by~\eqref{Eq:coefficients}. In particular, it is interesting to observe the interaction between element and patch bubbles. The coefficients of the element bubbles are defined to compensate the interior layers generated by each patch bubble, allowing an accurate representation of the behavior of the line $l_b(y)(1-x)$.

\begin{figure}[h]
\centering
    \begin{tikzpicture}[scale=0.45]
    
        \draw[line width=1pt, gray] (-8,0) -- (4,0);
        
        \draw[thick, b2, smooth] plot coordinates {
          (0,0) (0.45,1) (4,0)};
        \draw[thick, b2, smooth] plot coordinates {
            (-2,0) (-1.55,2) (2,0)};
        \draw[thick, b2, smooth] plot coordinates {
            (-4,0) (-3.55,3) (0,0)};
        \draw[thick, b2, smooth] plot coordinates {
            (-6,0) (-5.55,4) (-2,0)};
        \draw[thick, b2, smooth] plot coordinates {
            (-8,0) (-7.55,5) (-4,0)};
        
        \draw[thick, gray, smooth] plot coordinates {
            (0,0) (0.4,-0.9) (2,0)};
        \draw[thick, gray, smooth] plot coordinates {
            (-2,0) (-1.6,-1.9) (0,0)};    
        \draw[thick, gray, smooth] plot coordinates {
            (-4,0) (-3.6,-2.9) (-2,0)};
        \draw[thick, gray, smooth] plot coordinates {
            (-6,0) (-5.6,-3.9) (-4,0)};
        \draw[thick, gray, smooth] plot coordinates {
            (-8,0) (-7.6,-4.9) (-6,0)};

        \draw[thick, bT] (4,0) -- (-6,3);   
        \draw[thick, bT] (-8,0) -- (-6,3);   
         
        \draw[thick, gray] (4,0.15) -- (4,-0.15);
        \draw[thick, gray] (2,0.15) -- (2,-0.15);
        \draw[thick, gray] (0,0.15) -- (0,-0.15);
        \draw[thick, gray] (-2,0.15) -- (-2,-0.15);
        \draw[thick, gray] (-4,0.15) -- (-4,-0.15);
        \draw[thick, gray] (-6,0.15) -- (-6,-0.15);
        \draw[thick, gray] (-8,0.15) -- (-8,-0.15);

        \node[b2] at (-4,5) {$\alpha_{i+\nicefrac12,1}\psi_{i+\nicefrac{1}{2},1}$};
        \node[gray] at (-1,-3) {$\alpha_{i,1}\psi_{i,1}$};
        \node[bT] at (3,2) {$
        \displaystyle
        \summLA[B]\psi_\lambda$};
        \end{tikzpicture}
    \caption{\small\label{fig:lateralView} Schematic lateral view of the bubbles on the parabolic boundary, with the coefficients given by~\eqref{Eq:coefficients}.}
\end{figure}
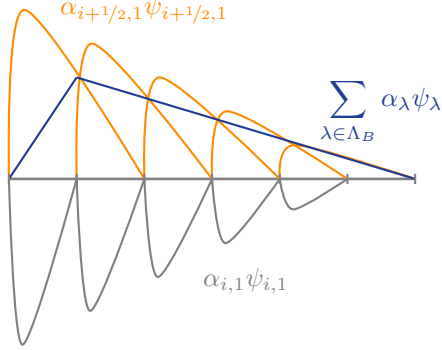
With this choice of coefficients we can prove the following result:

\begin{propo}\label{Propo:6}
Assume $u_B$ is defined by \eqref{Eq:psiI} in the parabolic boundary with the coefficients  
$\alpha_{i+1/2,1}, \alpha_{i,1}, \alpha_{i+1/2,N}, \alpha_{i,N}$ given by 
\eqref{Eq:coefficients}. Then,

$$
u_B^0 = \sum_{i=1}^{N-1} (\alpha_{i+1/2,1}\psi^0_{i+1/2,1} + \alpha_{i,1}\psi^0_{i,1}) = l_b(y)(1-x),
$$
for $h\leq x \leq 1$ and $0\leq y \leq h$
as well as 
$$
u_B^0 = \sum_{i=1}^{N-1} (\alpha_{i+1/2,N}\psi^0_{i+1/2,N} + \alpha_{i,N}\psi^0_{i,N}) = l_t(y)(1-x),
$$
for $h\leq x \leq 1$ and $1-h\leq y \leq 1$. Moreover
$$
(\alpha_{i-1/2,1}\psi^0_{i-1/2,1} + \alpha_{i,1}\psi^0_{i,1})_{|S_{i+1/2,1}} = 0,
$$
$$
(\alpha_{i-1/2,N}\psi^0_{i-1/2,N} + \alpha_{i,N}\psi^0_{i,N})_{|S_{i+1/2,N}} = 0,
$$
for $i=2,\dots N-1$.
\end{propo}

\begin{proof}
We prove the assertion at the bottom boundary; the top part can be proved analogously.

In $T_{1,1}$, $\psi^0(x,y) = 1 \psi^0_{3/2,1} = l_b(y)(1-x)$.
Now, let $i\in\{2,\dots N-1\}$. In $T_{i,1}$, by \eqref{Eq:psi} $\psi^0$  is given by the three contributions
$$
\begin{aligned}
  \psi^0 &= l_b(y)(\alpha_{i-1/2,1}\psi^0_{i-1/2,1} + \alpha_{i,1}\psi^0_{i,1} + \alpha_{i+1/2,1}\psi^0_{i+1/2,1}) \\
   & = l_b(y)\big[(i-1)(x_i + 2h -x) -2(i-1)(x_i+h-x) + i(x_i+h-x)\big] \\
   & = l_b(y)(x_i + ih -x) = l_b(y)(1-x).
\end{aligned}
$$
The second assertion holds by definition of the coefficients \eqref{Eq:coefficients} and the explicit
formula \eqref{Eq:psi}.
\end{proof}

\bigskip
\subsubsection{Interpolant in the elliptic outflow boundary}\label{Sec:elliptic}

At the left part of the domain (elliptic boundary)
it holds
$$
u^0 - u_L = 1-\frac{x}{h}, \qquad h < y < 1-h, \ 0 < x < h.
$$


    







The interpolant $u_B$ in the elliptic boundary has to fulfill
two conditions. First, the $\theta$-corrector of $u_B$ has to cancel the $\theta$-corrector
of $\ueps$, which requires
$$
u^0 - u_L = u_B^0,
\qquad h < y < 1-h, \ 0 < x < h.
$$
Then, since the $\theta$--correctors $\theta_{\ueps}$, $\theta_{B}$ for $\ueps$ and $u_B$, 
respectively are given by
\begin{equation}\label{eq:canceledTheta}
\begin{aligned}
  \theta_{\ueps}(x,y) & = -u^0(0,y) \exp(-x/\epsilon), \\
  \theta_{u_B}(x,y) & = -u_B^0(0,y) \exp(-x/\epsilon) = -u^0(0,y) \exp(-x/\epsilon),
\end{aligned}
\end{equation}
they cancel each other.

The second condition is that the $\varphi_B$-corrector for $u_B$ is small, since it cannot be bounded
like the one in the parabolic boundary layer, see Section \ref{Sec:correctors}. 
More precisely, we will choose $u_B$ in such a way
that $\varphi_{B} =$\est 
Both conditions are fulfilled if we choose
\begin{equation}\label{Eq:tildealpha}
- \alpha_{N,j} = \alpha_{N,j-1/2}=\alpha_{N,j+1/2}= N,\quad j=2,\dots N-1,
\end{equation}
and
$l\equiv 1$ in the definition of the residual free bubble $\psi_\lambda$, $\lambda \in \Lambda_L$,
as we will show now. 


Again, $\psi^0_{N,j}$ denotes the $0$-th order approximation of
$\psi_{N,j}$ and is given by 
\begin{equation*}
     - \partial_x \psi^0_{N,j} = 1   \quad\text{in } T_{N,j},\qquad
        \psi^0_{N,j} = 0 \text{\quad at $x=h$},
\end{equation*}
so that $\psi^0_{N,j} = h-x$,
and the same holds for $\psi_{N,j-1/2}^0,\psi_{N,j+1/2}^0$.

Then, with the choice \eqref{Eq:tildealpha} for the coefficients,
observing that, 
$$
u_B = \alpha_{N,j} \psi_{N,j} + \alpha_{N,j-1/2} \psi_{N,j-1/2}
    + \alpha_{N,j+1/2} \psi_{N,j+1/2},
    \quad\text{in $T_{N,j}$,}
$$
$j=2,\dots N-1$, we have
$$
\begin{aligned}
u_B^0 &= \alpha_{N,j} \psi_{N,j}^0 +\alpha_{N,j-1/2} \psi_{N,j-1/2}^0
    + \alpha_{N,j+1/2} \psi_{N,j+1/2}^0 \\
 & = (\alpha_{N,j}  + \alpha_{N,j-1/2} + \alpha_{N,j+1/2} )(h-x) =
             N(h-x) = \frac{h-x}{h},
\end{aligned}
$$
so that the first condition is fulfilled.

\bigskip
Regarding the second condition, note that the $\varphi$--correctors for 
$\psi_{N,j-1/2}$ and $\psi_{N,j+1/2}$ fulfill
$$
 (\varphi_{\psi_{N,j-1/2}} )_{|S_{N,j-1/2}} = \text{\est },\qquad
 (\varphi_{\psi_{N,j+1/2}} )_{|S_{N,j+1/2}} = \text{\est }
$$
and 
$$
(\alpha_{N,j} \psi_{N,j}^0 + \alpha_{N,j-1/2}\psi_{N,j-1/2}^0)_{|S_{N,j+1/2}} = 0,
$$
$$
(\alpha_{N,j} \psi_{N,j}^0 + \alpha_{N,j+1/2}\psi_{N,j+1/2}^0)_{|S_{N,j-1/2}} = 0,
$$
so that the $\varphi$--corrector $\varphi_B$ for $u_B$ in $T_{N,j}$ satisfies
$
\varphi_{B} = \est.
$

\subsubsection{Interpolant in the outflow corners}

It remains to consider the elements $T_{N,1}$ and $T_{N,N}$ at outflow corners.
We focus on  $T_{N,1}$, the element $T_{N,N}$ is treated in exactly the same way.

In $T_{N,1}$ it holds
$$
u^0 - u_L = 1-x - \frac{1-h}{h^2}xy,
$$
and we have
$$
\begin{aligned}
u_B &= \alpha_{N,1}\psi_{N,1} + \alpha_{N-1/2,1}\psi_{N-1/2,1} + \alpha_{N,1/2}\psi_{N,1/2} \\
 &= (2-2N)\psi_{N,1} + (N-1)\psi_{N-1/2,1} + N\psi_{N,1/2},
\end{aligned}
$$
where we have used the coefficient values given by \eqref{Eq:coefficients} and \eqref{Eq:tildealpha}.

Note that the $\theta$-corrector for $\alpha_{N,1}\psi_{N,1} + \alpha_{N-1/2,1}\psi_{N-1/2,1}$ is zero as in
the parabolic boundary layer. 

The difference $u^0 - u_L - \psi^0_I$ is again calculated as
\begin{equation}\label{Eq:r}
u^0 - u_B - \psi^0_I = 0, \qquad\text{in $T_{N,1}$}.
\end{equation}



\bigskip
We summarize the properties of our construction in the following proposition: 
\begin{propo}\label{prop:bilinear0}
Let $u_L$ be the bilinear interpolation of $u^0$ in the interior extended to $\bar{\Omega}$ with zero boundary values on
$\partial\Omega$ and let $u_B$ the bubble interpolant given by \eqref{def:Lambda+u_B}, with coefficients defined in \eqref{Eq:coefficients} and \eqref{Eq:tildealpha}. Then
$$
u^0 - u_L - u_B^0 = 0.
$$
\end{propo}
\subsubsection{Scaling}

One main tool in the subsequent analysis is scaling back and forth
to a reference situation.
Let $T_{i,j}$, $i,j \in\{1,\dots, N\}$ be an element of the mesh. Define
$$
\Phi_{i,j} :\hat{T}:=[0,1]^2 \rightarrow T_{i,j}, 
\qquad
\Phi_{i,j}(\hat{x},\hat{y}) = (x_i + h\hat{x},y_j + h\hat{y})^T,
$$
and associate $(x,y)=\Phi_{i,j}(\hat{x},\hat{y})$.
We also use the notation that for a function $v$ on $T_{i,j}$,
$\hat{v}(\hat{x},\hat{y})=v(x,y)$.

On $\hat{T}$ the element bubbles fulfill the following equation:
\[
-\frac{\epsilon}{h^2} \hat{\Delta} \hat{\psi} - \frac{1}{h} \partial_{\hat{x}}\hat{\psi} = \hat{l}(\hat{y})\quad \text{in }\hat{T},\qquad
   \hat{\psi} = 0\quad\text{on }\partial\hat{T}.
\]
For the patch bubbles we have the same scaling but rather for $\hat{T} = (0,2)\times (0,1)$ or
$\hat{T} = (0,1)\times (0,2)$.

Multiplying the above equation by $h$ we get, 
with $\hat{\epsilon}=\frac{\epsilon}{h}$,
\begin{equation}\label{eq:defHAtPsi}
-\hat{\epsilon} \hat{\Delta} \hat{\psi} - \partial_{\hat{x}}\hat{\psi} = h\hat{l}(\hat{y})\quad \text{in }\hat{T},\qquad
   \hat{\psi} = 0\quad\text{on }\partial\hat{T}.
\end{equation}

In the next section we apply the results of the asymptotic expansion to the various bubble functions
on the reference configuration. There the bubbles fulfill the same equation \eqref{Eq:problem} $\epsilon$ replaced by $\hat{\epsilon}$, and with a different right-hand side $f$.

\subsubsection{Estimating the correctors for the bubble functions}\label{Sec:correctors}

The bubble interpolant $u_B$ is defined in~\eqref{def:Lambda+u_B}, with the coefficients given by ~\eqref{Eq:coefficients} and~\eqref{Eq:tildealpha}.
In this subsection, 
we study the difference between $u_B - \uBas$.
Recalling the definition of \uBas from~\eqref{def:uB_ass},
\[
\uBas = u_B^0 + \underbrace{\theta_B+ \varphi_B +\xi_B + \zeta_B}_{\text{correctors for }u_B},
\]
we derive bounds for each corrector.

First, let us analyze the correctors corresponding to the ordinary boundary layer at $x=0$.
By construction, the correctors $\theta$ and $\theta_B$ associated for $\ueps$ and $u_B$ respectively, cancel each other at the outflow boundary and at the outflow corners, where $\theta$ is active, as mentioned in~\eqref{eq:canceledTheta}.

On the other hand, the sum of the correctors $\theta_\lambda$ vanishes in the bottom and top boundary layers. Indeed, let us consider $T_{i,1}$ for $i=1,\ldots,N-1$, and two of the bubbles defined in this element:
\begin{equation*}
    \alpha_{i,1}\psi_{i,1} + \alpha_{i-1/2,1}\psi_{i-1/2,1}.     
\end{equation*}
By their definition~\eqref{def:theta}, the corresponding correctors $\theta_\lambda$ satisfy
\begin{equation*}
    \alpha_{i,1}\theta_{i,1} + \alpha_{i-1/2,1}\theta_{i-1/2,1} = 
    -(\alpha_{i,1}\psi_{i,1}^0 + \alpha_{i-1/2,1}\psi_{i-1/2,1}^0 ) (x_i,y)\exp(-x/\epsilon).
\end{equation*}
The second statement of Proposition~\ref{Propo:6} ensures that the factor in parentheses is zero.

The following Lemma provides a bound for the remaining correctors of $u_B$ in the $\epsilon-$norm.

\begin{lemma}
\label{Lemma:boundBubbleCorr}
Let $\varphi_B, \xi_B, \zeta_B$ denote the correctors corresponding to $u_B$ given by \eqref{eq:defBubbCorr}.
Then,
$$
\|\varphi_B\|_\epsilon  +  \|\xi_B\|_\epsilon + \|\zeta_B\|_\epsilon  \lesssim  (\epsilon/h)^{1/4}.
$$
\end{lemma}

\begin{proof}
Let $\chi^B\in\{\varphi_B, \xi_B,\zeta_B \}$ a corrector of $u_B$. Then,
\begin{equation*}
\begin{aligned}
    \|\chi^B\|_\epsilon^2 &= \sum_{T\in\Omega \setminus \mathaccent 23 \Omega} \|\chi^B\|_{\epsilon,T}^2 
    \le
    \sum_{T\in\Omega \setminus \mathaccent 23 \Omega}\sum_{\lambda \in \Lambda(T)}\alpha_\lambda^2\|\chi_\lambda\|_{\epsilon,T}^2
    \\
    &\le \frac{4}{h^2}\sum_{T\in\Omega \setminus \mathaccent 23 \Omega}\sum_{\lambda \in \Lambda(T)}\|\chi_\lambda\|_{\epsilon,T}^2 
    \lesssim \frac{1}{h^2} N  h^3 (\epsilon/h)^{1/2}
    =(\epsilon/h)^{1/2}
    ,
\end{aligned}
\end{equation*}
where the last inequality follows from Lemma~\ref{Temam:BubbleCorrectorsBound}.
\end{proof}

We also need to bound the correctors in the full $\vvvert \cdot\vvvert $ norm. 
This part is rather technical and is thus split into several lemmas.
Recalling from~\eqref{eq:||| norm} that $
\vvvert  v\vvvert   := \|v\|_\epsilon + 
\sup_{0\not= v_h\in V_h} \frac{a(v,v_h)}{\|v_h\|_\epsilon},
$
we observe that the above correctors have been already bounded in the $\|\cdot \|_\epsilon$--norm in Lemma~\ref{Lemma:boundBubbleCorr}. 
We need to bound the additional term for each corrector, and this will be proved in each of the following lemmas.

\begin{lemma}\label{Lemma:10}
Let $\varphi_B$ denote the $\varphi$-corrector corresponding to $u_B$ given by \eqref{eq:defBubbCorr}.
Then,
$$
 a(\varphi_B,v_h) \lesssim (\epsilon/h)^{1/4} \|v_h\|_\epsilon, \qquad \text{for all $v_h \in V_h$}.
$$
\end{lemma}

\begin{proof}
Recall that
\[\varphi_B = \summLA[B]\varphi_\lambda+ \summLA[T]\varphi_\lambda+
\summLA[L]\varphi_\lambda,\]
and let us focus first on the integrals corresponding to the indices in $\Lambda_B$. 
Let $v_h\in V_h$, then
\begin{equation*}
    \begin{aligned}
    a(\summLA[B] \varphi_\lambda,v_h) 
    &= \summLA[B] a(\varphi_\lambda,v_h)\\
    &=  \summLA[B]\left( \epsilon \int_{\Omega_\lambda} \nabla\varphi_\lambda \nabla (v_h e^x)
        - \int_{\Omega_\lambda} \partial_x \varphi_\lambda v_h e^x \right).
    \end{aligned}
\end{equation*}

Let $\lambda\in\Lambda_B$ and $\Omega_\lambda:=\supp\{\varphi_\lambda\}$. It follows that
\begin{equation*}
    \begin{aligned}
     \big( \epsilon &\int_{\Omega_\lambda}\nabla\varphi_\lambda \nabla (v_h e^x)
        - \int_{\Omega_\lambda}\partial_x \varphi_\lambda v_h e^x \big) 
        = \big( \epsilon \int_{\hat{\Omega}_\lambda}\hat{\nabla}\hat{\varphi}_\lambda\hat{\nabla}(\hat{v}_h e^x)
              - h \int_{\hat{\Omega}_\lambda} \partial_{\hat{x}} \hat{\varphi}_\lambda \hat{v}_h e^x   \big) \\
     &=  h(\hat{\epsilon} \int_{\hat{\Omega}_\lambda}\partial_{\hat{y}} \hat{\varphi}_\lambda\partial_{\hat{y}}\hat{v}_h e^x
              -  \int_{\hat{\Omega}_\lambda} \partial_{\hat{x}} \hat{\varphi}_\lambda \hat{v}_h e^x) + \epsilon \int_{\hat{\Omega}_\lambda}
                  \partial_{\hat{x}} \hat{\varphi}_\lambda\partial_{\hat{x}}(\hat{v}_he^x) \\
     &=  h\Big(\underbrace{\hat{\epsilon} \int_{\hat{\Omega}_\lambda}-\partial_{\hat{y}\hat{y}} \hat{\varphi}_\lambda\hat{v}_h e^x
              -  \int_{\hat{\Omega}_\lambda} \partial_{\hat{x}} \hat{\varphi}_\lambda \hat{v}_h e^x}_{=0}\Big) + \epsilon \int_{\hat{\Omega}_\lambda}
               \partial_{\hat{x}} \hat{\varphi}_\lambda\partial_{\hat{x}}(\hat{v}_he^x)  + \text{e.s.t}\int_{\hat{\Omega}_\lambda}\hat{v}_h.
    \end{aligned}
\end{equation*}
In the last step, we used the definition of $\varphi_\lambda$, together with the fact that $\partial_{\hat{y}}\hat{\varphi}_\lambda|_{\hat{y}=1} =$ \est and $\hat{v}_h|_{\hat{y}=0}=0$.

Moreover, 
\begin{equation*}
    \begin{aligned}
\epsilon \int_{\hat{\Omega}_\lambda} \partial_{\hat{x}} \hat{\varphi}_\lambda\partial_{\hat{x}}(\hat{v}_he^x) 
& \lesssim
     \epsilon \|\partial_{\hat{x}}\hat{\varphi}_\lambda\|_{\hat{\Omega}_\lambda} \|\partial_{\hat{x}} (\hat{v}_he^x) \|_{\hat{\Omega}_\lambda}\\
     & \lesssim
    \epsilon \|\partial_x \varphi_\lambda\|_{\Omega_\lambda}\|\partial_x (v_he^x) \|_{\Omega_\lambda}\\
     & \lesssim
     \epsilon \|\partial_x \varphi_\lambda\|_{\Omega_\lambda}(\|\partial_x v_h) \|_{\Omega_\lambda}+ \|v_h \|_T)\\
      &=  \sqrt{\epsilon} \|\partial_x \varphi_\lambda\|_{\Omega_\lambda}\sqrt{\epsilon}( \|\partial_x v_h \|_{\Omega_\lambda}+
    \|v_h \|_T).
    \end{aligned}
\end{equation*}

Therefore, summing up all elements and using Lemma~\ref{Lemma:boundBubbleCorr}, we obtain
\begin{equation*}
 \begin{aligned}
a(\summLA[B]\varphi_\lambda,v_h) &\lesssim 
    \summLA[B]\left(\sqrt{\epsilon} \|\partial_x \varphi_\lambda\|_{\Omega_\lambda}\sqrt{\epsilon}( \|\partial_x v_h \|_{\Omega_\lambda}+
        \|v_h \|_{\Omega_\lambda})
              + \text{e.s.t}\int_{\hat{\Omega}_\lambda}\hat{v}_h\right)\\
    &\lesssim \left(
    \sum_{\lambda\in\Lambda_{B}}\alpha_\lambda^2\epsilon
    \|\partial_x \varphi_\lambda\|_{\Omega_\lambda}^2
    \right)^{1/2}\|v_h\|_\epsilon 
    + 
    \text{e.s.t} \, \| v_h \| 
    \\
    & \le \hat{\epsilon}^{1/4}\|v_h\|_{\epsilon}.  
 \end{aligned}
\end{equation*}
We can obtain an analogous bound for $a(\summLA[T]\varphi_\lambda,v_h)$, whence
\begin{equation*}
 \begin{aligned}
a(\summLA[B]\varphi_\lambda+\summLA[T]\varphi_\lambda,v_h) 
&\lesssim 
     \hat{\epsilon}^{1/4}\|v_h\|_{\epsilon}.  
 \end{aligned}
\end{equation*}
The $\varphi$--correctors at the outflow boundary are \est\ by construction, see Section \ref{Sec:elliptic}; and the assertion of the lemma thus follows.
\end{proof}

\begin{lemma}\label{Lemma:11}
Let $\xi_B$ denote the $\xi$-corrector corresponding to $u_B$ given by \eqref{eq:defBubbCorr}.
Then,
$$
 a(\xi_B,v_h) \lesssim (\epsilon/h)^{1/4} \|v_h\|_\epsilon, \qquad \text{for all $v_h \in V_h$}.
$$
\end{lemma}

\begin{proof}
Recall that
\[\xi_B = \summLA[B]\xi_\lambda+ \summLA[T]\xi_\lambda+
\summLA[L]\xi_\lambda,\]
and let us focus first on the integrals corresponding to the indices in $\Lambda_B$. 
Let $v_h\in V_h$, then
\begin{equation*}
\begin{aligned}
a\Big(\summLA[B]\xi_\lambda,v_h\Big) &=
        \summLA[B]
        \left( \epsilon \int_{\Omega_\lambda} \nabla\xi_\lambda \nabla (v_h e^x)
        - \int_{\Omega_\lambda} \partial_x \xi_\lambda v_h e^x \right)\\
        &\le
        \summLA[B]
        \left( \sqrt{\epsilon} \| \nabla\xi_\lambda\|_{\Omega_\lambda} \sqrt{\epsilon}\|\nabla (v_h e^x)\|_{\Omega_\lambda}
        - \int_{\Omega_\lambda} \partial_x \xi_\lambda v_h e^x \right)\\
        &\le
        \|\xi_B\|_\epsilon
    \|v_h\|_\epsilon- \summLA[B]\int_{\Omega_\lambda} \partial_x \xi_\lambda v_h e^x .
\end{aligned}
\end{equation*}

The first term is bounded as required by Lemma~\ref{Lemma:boundBubbleCorr}. We now turn our attention to the second term. Integrating by parts in $x$, we obtain:
$$
  \int_{\Omega_\lambda} \partial_x \xi_\lambda v_h e^x =   -\int_{\Omega_\lambda} \xi_\lambda \partial_x v_h e^x - \int_{\Omega_\lambda} \xi_\lambda  v_h e^x,
$$
where the boundary terms vanish because $\xi=0$ on the left and right sides of $\Omega_\lambda$.
The first term on the right-hand side is treated as follows:
$$
\begin{aligned}
   \Big|\int_{\Omega_\lambda} \xi_\lambda \partial_x v_h e^x \Big| 
   &= h \Big|
   \int_{\hat{\Omega}_\lambda} \hat{\xi}_\lambda \partial_{\hat{x}} \hat{v}_h e^x\Big|
   \lesssim
     h \|\hat{\xi}_\lambda\|_{\hat{\Omega}_\lambda} \| \partial_{\hat{x}} \hat{v}_h \|_{\hat{\Omega}_\lambda} \\
   &\lesssim  h^2 \hat{\epsilon} \|\partial_{\hat{x}} \hat{v}_h \|_{\hat{\Omega}_\lambda} 
   =
   h\epsilon \, \|\partial_x v_h \|_{\Omega_\lambda},
\end{aligned}
$$
where, in the last inequality, we have used Lemma \ref{Temam:Ap}.
Similarly,
$$
\begin{aligned}
   \Big|\int_{\Omega_\lambda} \xi_\lambda  v_h e^x  \Big|&= h^2 \Big|\int_{\hat{\Omega}_\lambda} \hat{\xi}_\lambda  \hat{v}_h e^x\Big|\lesssim
    h^2 \|\hat{\xi}_\lambda\|_{\hat{\Omega}_\lambda} \, \| \hat{v}_h \|_{\hat{\Omega}_\lambda}
  \lesssim  h^3 \hat{\epsilon} \|\hat{v}_h \|_{\hat{\Omega}_\lambda} 
  =
   h\epsilon \, \| v_h \|_{\Omega_\lambda}.
\end{aligned}
$$

Then, since $|\alpha_\lambda |\le 2N$, $\#\Lambda_B \lesssim N$ and $N = 1/h$,
$$
\begin{aligned}
   \summLA[B]\int_{\Omega_\lambda} |\partial_x \xi_\lambda v_h e^x|
    & \lesssim 
    \summLA[B] \, h \, \epsilon \, \big(\|\partial_x v_h \|_{\Omega_\lambda}  
           +\| v_h \|_{\Omega_\lambda}\big) \\   
   &\lesssim 
   \Big(\summL[B] \epsilon\Big)^{1/2} \Big(\summL[B]\epsilon \|\partial_x v_h \|_{\Omega_\lambda}^2 + \| v_h \|_{\Omega_\lambda}^2\Big)^{1/2}\\
   &\lesssim
   \Big(\frac\epsilon h\Big)^{1/2} \, \|v_h\|_\epsilon.
\end{aligned}
$$

Analogously we can obtain a similar bound for $a\Big(\summLA[T]\xi_\lambda,v_h\Big) $ and  also for $a\Big(\summLA[L]\xi_\lambda,v_h\Big) $, whence the assertion of the lemma follows.
\end{proof}

\begin{lemma}\label{Lemma:12}
Let $\zeta_B$ denote the $\zeta$-corrector corresponding to $u_B$ given by \eqref{eq:defBubbCorr}. There is $\hat{\epsilon}_0, 0 < \hat{\epsilon}_0 < 1$
such that
$$
 a(\zeta_B,v_h) \lesssim (\epsilon/h)^{1/4} \|v_h\|_\epsilon, \qquad \text{for all $v_h \in V_h$}
$$
for all $0<\epsilon/h <\hat{\epsilon}_0$.
\end{lemma}

\begin{proof}
Recall that
\[\zeta_B = \summLA[B]\zeta_\lambda+ \summLA[T]\zeta_\lambda+
\summLA[L]\zeta_\lambda,\]
and let us focus first on the integrals corresponding to the indices in $\Lambda_B$. 
Let $v_h\in V_h$, then, analogously as before,
\begin{equation}\label{Eq:79}
a\big(\summLA[B]\zeta_\lambda,v_h\big)  \lesssim 
\big\|\zeta_B\big\|_\epsilon \big\|v_h\big\|_\epsilon + \summLA[B] \Big|\int_{\Omega_\lambda} \partial_x \zeta_\lambda v_h e^x\Big|.
\end{equation}
Like above, by scaling and integration by parts we get
\begin{equation}\label{eq:xi_boundary}
\begin{aligned}
   \int_{\Omega_\lambda} \partial_x \zeta_\lambda v_h e^x 
  &= -h\int_{\hat{\Omega}_\lambda} \hat{\zeta}_\lambda \partial_{\hat{x}}(\hat{v}_h e^x)
  - h\int_{\Gammaleft} \hat{\zeta}_\lambda \hat{v}_h e^x  + h\int_{\Gammaright} \text{\est } \hat{v}_h,
\end{aligned}
\end{equation}
where we consider $\Gammaleft$ and $\Gammaright$, the left and right boundary of the domain $\hat{\Omega}_\lambda$. That is,
\begin{equation*}
 \Gammaleft := \left\{ (\hat{x}, \hat{y}) \in \partial \hat{\Omega}_\lambda \,:\, \hat{x} = x_{\min} \right\}\text{\quad and \quad}
 \Gammaright := \left\{ (\hat{x}, \hat{y}) \in \partial \hat{\Omega}_\lambda \,:\, \hat{x} = x_{\max} \right\},
\end{equation*}
where $x_{\min}$ and $x_{\max}$ denotes the minimal and maximal $\hat{x}$-coordinate of $\hat{\Omega}_\lambda$, correspondingly.

Since $\partial_{\hat{x}} e^x = h e^x$, the first term of the last line of~\eqref{eq:xi_boundary} can be estimated as
\begin{equation}\label{Eq:77}
\begin{aligned}
\Big|h\int_{\hat{\Omega}_\lambda} \hat{\zeta}_\lambda \partial_{\hat{x}}(\hat{v}_h e^x)\Big|
 &\leq \Big|h\int_{{\hat{\Omega}_\lambda}} \hat{\zeta}_\lambda \partial_{\hat{x}}(\hat{v}_h) e^x\Big|
+ \Big|h^2\int_{\hat{\Omega}_\lambda} \hat{\zeta}_\lambda \hat{v}_h e^x\Big|
\\
&\le \|\hat{\zeta}_\lambda\|_{\hat{\Omega}_\lambda}
\big( h\big\|\partial_{\hat{x}} \hat{v}_h\big\|_{\hat{\Omega}_\lambda} +h^2 \big\| \hat{v}_h\big\|_{\hat{\Omega}_\lambda}\big)
\\
&\lesssim h \, \hat{\epsilon}^{3/4} \, \big(h\big\|\partial_{\hat{x}} \hat{v}_h\big\|_{\hat{\Omega}_\lambda} +h^2 \big\| \hat{v}_h\big\|_{\hat{\Omega}_\lambda}\big)
\\
&= h^2\hat{\epsilon}^{3/4} \big(\big\|\partial_x v_h\big\|_{\Omega_\lambda} +\big\|v_h\big\|_{\Omega_\lambda}\big),
\end{aligned}
\end{equation}
where in the last inequality, we have used Lemma~\ref{Temam:zetaBubble}.

The integral over \Gammaleft in~\eqref{eq:xi_boundary}   can be handled as follows:
Let ${v_h = q_h + b_h}$ with $q_h \in V_L$ being piecewise bilinear and $b_h \in V_B$  a bubble function.

Using Lemma~\ref{Temam:zetaBubble}, we obtain:
\begin{equation*}
\begin{aligned}
\Big|h\int_{\Gammaleft} \hat{\zeta}_\lambda \hat{v}_h e^x\Big| &= 
    \Big| h\int_{\Gammaleft} \hat{y} \hat{\zeta}_\lambda e^x \frac{\hat{v}_h}{\hat{y}} \Big|
  \lesssim h^2\hat{\epsilon}^{3/4} \bigg\| \frac{\hat{v}_h}{\hat{y}}\bigg\|_{\Gammaleft} 
  \\
  & \le h^2\hat{\epsilon}^{3/4} \Big(\bigg\| \frac{\hat{q}_h}{\hat{y}}\bigg\|_{\Gammaleft}+ \bigg\| \frac{\hat{b}_h}{\hat{y}}\bigg\|_{\Gammaleft}\Big).
\end{aligned}
\end{equation*}

For the first term, using Hardy's one-dimensional inequality and the equivalence of norms on finite dimensional spaces%
, we obtain that:
\begin{equation*}
 \begin{aligned}
  h^2\hat{\epsilon}^{3/4} \left\| \frac{\hat{q}_h}{\hat{y}}\right\|_{\Gammaleft}
 &\lesssim  h^2\hat{\epsilon}^{3/4} \| \partial_{\hat{y}} \hat{q}_h\|_{\Gammaleft} 
 \lesssim h^2\hat{\epsilon}^{3/4} \| \hat{\nabla} \hat{q}_h\|_{\hat{\Omega}_\lambda}
 \lesssim  h^2\hat{\epsilon}^{3/4} \| \nabla q_h\|_{\Omega_\lambda}.
 \end{aligned}   
\end{equation*}

For the second term, due to the support of the bubbles that define $V_B$, we have $\hat{b}_h = c \, \hat{b}_{S_0}$ at $\hat{x} = x_{\min}$, where $\hat{b}_{S_0}$ denotes the patch bubble associated with the side $\hat{x} = x_{\min}$, for some constant $c$. Applying Proposition~\ref{propo: bh_inf_bound}, we obtain:

\begin{equation*}
      h^2\hat{\epsilon}^{3/4} \bigg\| \frac{\hat{b}_h}{\hat{y}}\bigg\|_{\Gammaleft}
      \lesssim 
      h^2\hat{\epsilon}^{3/4} \| \nabla  \hat{b}_h\|_{\hat{\Omega}_\lambda}.
\end{equation*}

Then, using Proposition \ref{Propo:AngleSpace2} 
one finds $\hat{\epsilon}_0, 0 < \hat{\epsilon}_0 < 1$ such that
\begin{equation}\label{Eq:78}
    |h\int_{\Gammaleft} \hat{\zeta}_\lambda \hat{v}_h e^x| \lesssim
     h^2\hat{\epsilon}^{3/4} \| \nabla v_h\|_{\Omega_\lambda}.
\end{equation}
for all $0<\epsilon/h < \hat{\epsilon}_0$.
For the last term in~\eqref{eq:xi_boundary}, by equivalence of norms, we obtain
\begin{equation}
\label{Eq:80}
    h\int_{\Gammaright} \text{\est } \hat{v}_h
    \leq
    h\est \int_{\hat{\Omega}_{\lambda}}\hat{v}_h,
\end{equation}

By combining~\eqref{Eq:79}, Lemma~\ref{Lemma:boundBubbleCorr}, and equations~\eqref{Eq:77},\eqref{Eq:78},\eqref{Eq:80}, using that $|\alpha_\lambda |\le 2N$, $\#\Lambda_B \lesssim N$, and $N = 1/h$, we conclude that

$$
\begin{aligned}
a\Big(\summLA[B] \zeta_\lambda,v_h\Big) 
    &\lesssim 
    \hat{\epsilon}^{1/4} \|v_h\|_\epsilon
 +  \summLA[B] h^2 \hat{\epsilon}^{3/4} (\|\nabla v_h\|_{\Omega_\lambda} +  \|v_h\|_{\Omega_\lambda}) \\
  & \lesssim
  \hat{\epsilon}^{1/4} \|v_h\|_\epsilon
 +  \hat{\epsilon}^{1/4} \Big(\summL[B] h\Big)^{1/2}  \Big(\summL[B] \epsilon \|\nabla v_h\|_{\Omega_\lambda}^2 + \|v_h\|_{\Omega_\lambda}^2\Big)^{1/2}  \\
   & \lesssim \hat{\epsilon}^{1/4} \|v_h\|_\epsilon.
\end{aligned}
$$

Analogously we can prove the same bound for $a\Big(\summLA[B] \zeta_\lambda,v_h\Big) $, whence
$$
\begin{aligned}
a\Big(\summLA[B] \zeta_\lambda+\summLA[T] \zeta_\lambda,v_h\Big) 
   & \lesssim \hat{\epsilon}^{1/4} \|v_h\|_\epsilon.
\end{aligned}
$$

We finally consider the outflow boundary. That is, we now consider the indices $\lambda\in\Lambda_L$, and
proceed again as for the parabolic boundary layer. The estimate is even simpler, 
since the boundary term
$$
h\int_{\Gammaleft} \hat{\zeta}_\lambda \hat{v}_h e^x = 0,
$$ because $v_{h|\partial\Omega}=0$.
%
\end{proof}

As a consequence of the last three lemmas and Lemma~\ref{Lemma:boundBubbleCorr} we arrive at the following result.

\begin{lemma}\label{Lemma:13}
Let $\varphi_B, \xi_B, \zeta_B$ denote the correctors corresponding to $u_B$ given by \eqref{eq:defBubbCorr}. There is  $\hat{\epsilon}_0, 0 < \hat{\epsilon}_0 < 1$ such that
$$
\vvvert \varphi_B\vvvert   +  \vvvert \xi_B\vvvert  + \vvvert \zeta_B\vvvert   \lesssim  (\epsilon/h)^{1/4}
$$
for all $0<\epsilon/h < \hat{\epsilon}_0$.
\end{lemma}

\begin{lemma}\label{Lemma:8}
For the difference between $u_B$ defined by~\eqref{def:Lambda+u_B} with the coefficients given by~\eqref{Eq:coefficients}, \eqref{Eq:tildealpha}  and its asymptotic expansion $\uBas$ defined by~\eqref{def:uB_ass} we have the following bound:
$$
\vvvert  u_B - \uBas \vvvert  \lesssim (\epsilon/h)^{1/2}.
$$
\end{lemma}

\begin{proof}
First we bound $u_B - u_B^{as}$ in the $\|\cdot \|_\epsilon$--norm. Using Lemma~\ref{Temam:ErrorBound}, we obtain
$$
\begin{aligned}
\|u_B - u_B^{as} \|_\epsilon^2 & 
\le
\summLA^2 \|\psilambda - \psilambdaas \|_{\epsilon,\Omega_\lambda}^2 \\
& =
\summLA^2(\epsilon \|\nabla(\psilambda - \psilambdaas) \|_{\Omega_\lambda}^2 + \| \psilambda - \psilambdaas \|_{\Omega_\lambda}^2)\\
&\lesssim
\summLA^2(h\hat{\epsilon} \|\hat{\nabla}(\hat{\psi}_\lambda - (\hat{\psi}_\lambda)_{as}) \|_{\hat{\Omega}_\lambda}^2 + h^2\| (\hat{\psi}_\lambda - (\hat{\psi}_\lambda)_{as}) \|_{\hat{\Omega}_\lambda}^2)\\
&\lesssim
\summLA^2h\| \hat{\psi}_\lambda - (\hat{\psi}_\lambda)_{as}\|_{\hat{\epsilon},\hat{\Omega}_\lambda}^2 \lesssim
\summLA^2h h^2 \hat{\epsilon}^2
\lesssim
\hat{\epsilon}^2.
\end{aligned}
$$

Next, take $ v_h\in V_h$,
$$
 a(u_B - u_B^{as},v_h)  \lesssim \|u_B - u_B^{as} \|_\epsilon \|v_h\|_\epsilon
   - \summLA \int_{\Omega_\lambda} \partial_x (u_B - u_B^{as}) v_h.
$$
The last term on the right hand side can be estimated as follows:
$$
\begin{aligned}
    - \summLA \int_{\Omega_\lambda} \partial_x (u_B - u_B^{as}) v_h & = 
          - \summLA h \int_{\hat{\Omega}_\lambda} \partial_{\hat{x}} (\hat{\psi}_\lambda - (\hat{\psi}_\lambda)_{as}) \hat{v}_h \\
           & = \summL \alpha_\lambda h \int_{\hat{\Omega}_\lambda}  (\hat{\psi}_\lambda - (\hat{\psi}_\lambda)_{as}) \partial_{\hat{x}} \hat{v}_h 
            + \text{\est } \|\partial_{\hat{x}}\hat{v}_h\|_{\hat{\Omega}_\lambda},
\end{aligned}
$$
since $(\hat{\psi}_\lambda - (\hat{\psi}_\lambda)_{as})_{|\partial\hat{T}}=$ \est Further, by Lemma~\ref{Temam:ErrorBound} we get:
$$
\begin{aligned}
            \summL \alpha_\lambda h\int_{\hat{\Omega}_\lambda}  (\hat{\psi}_\lambda - (\hat{\psi}_\lambda)_{as}) \partial_{\hat{x}} \hat{v}_h 
                &\lesssim
                \summL  \|\hat{\psi}_\lambda - (\hat{\psi}_\lambda)_{as}\|_{\hat{\Omega}_\lambda}
                   \|\partial_x v_h\|_{\Omega_\lambda} \\
           & \lesssim 
           \summL  h\hat{\epsilon}  \|\partial_x v_h\|_{\Omega_\lambda}
            = \summL \sqrt{h} \sqrt{\hat{\epsilon}} \sqrt{\epsilon} \|\partial_x v_h\|_{\Omega_\lambda} \\
           & \lesssim \hat{\epsilon}^{1/2} (\summL h)^{1/2} \| v_h\|_\epsilon \lesssim 
            \hat{\epsilon}^{1/2} \| v_h\|_\epsilon. 
\end{aligned}
$$
The result now follows from the definition of the $\vvvert \cdot \vvvert $--norm.
\end{proof}

\subsection{Proof of main error estimate}

Putting all things from above together we finally get:

\begin{thm}(Final error estimate)\label{Thm:final error estimate}
Let the notations and assumptions from the previous sections hold. 
Then there is $\hat{\epsilon}_0<1$, such that
$$
\vvvert \ueps - u_h\vvvert  \lesssim (\epsilon/h)^{1/4}.
$$
for all $0< \epsilon/h < \hat{\epsilon}_0.$
\end{thm}

\begin{proof}
By Cea's Lemma \ref{Lemma:Cea} we can estimate
$$
\vvvert \ueps - u_h\vvvert  \lesssim \vvvert \ueps - \ueps_I \vvvert  \lesssim  \vvvert  \ueps - \uas\vvvert  + \vvvert  \uas - \ueps_I\vvvert.
$$

From Lemma \ref{Lemma:4} we infer
$$
\vvvert  \ueps - \uas\vvvert  \lesssim \sqrt{\epsilon}.
$$

With the choice of the interpolant $\ueps_I = u_L + u_B$, $\uas = u^0 + \text{correctors for } \ueps$  and $\uBas = u_B^0 + \text{correctors for } u_B$
we get by Proposition~\ref{prop:bilinear0}, Lemma~\ref{Lemma:5}, Lemma~\ref{Lemma:10} and Lemma~\ref{Lemma:8}
$$
\begin{aligned}
\vvvert  \uas - \ueps_I\vvvert &=  \vvvert \uas - (u_L + u_B) +\uBas - \uBas\vvvert  \\
       & \leq \vvvert u^0 - (u_L + u_B^0)\vvvert   + \vvvert \text{correctors for }\uas\vvvert  + \vvvert \text{correctors for }u_B\vvvert  \\
       &\qquad + \vvvert u_B^{as} - u_B\vvvert \\
          &= \vvvert \text{correctors for }\uas\vvvert  + \vvvert \text{correctors for }u_B\vvvert  +
      \vvvert u_B^{as} - u_B\vvvert \\
    &\lesssim 
    \epsilon^{1/4} + \hat{\epsilon}^{1/4} + \hat{\epsilon}^{1/2},
\end{aligned}
$$
and the theorem is proved.
\end{proof}

\section{Numerical results}\label{Sec:numerics}
In this section we present an experimental study comparing the performance of three methods: BMZ, the newly proposed patch bubble method; RFB, the classical Residual-Free Bubbles method~\cite{Brezzi1994}; and RFBe, the enhanced Residual-Free Bubbles method proposed by Cangiani-S\"uli~\cite{Cangiani2005long}. Several norms will be considered in order to capture different aspects of the approximation error.
The results shown for RFBe are obtained using our recursive formulation of the algorithm in the calculus of their patch bubbles, which provides more accurate approximations than the original implementation presented in~\cite{Cangiani2005short}.
A sequence of uniform partitions of $\Omega=(0,1)^2$ with $N\times N$ elements is considered, with $N$ increasing.

We solve
\begin{equation}
    \label{eq:NumEx}
-\epsilon \Delta u - u_x = f \text{ in $\Omega$},
\qquad u=0  \text{ on $\partial\Omega$}.
\end{equation}
with $\epsilon = 10^{-6}$ and $f$ such that the exact solution is given by
\begin{equation*}
 u(x,y) = \psi(x)\,\psi(y^2),
 \qquad \psi(x) = 1-x - \frac{\exp(-x/\epsilon) - \exp(-1/\epsilon) }{  1 - \exp(-1/\epsilon) }.
\end{equation*}
which satisfies $-\epsilon \psi'' - \psi' = 1$, $\psi(0)=\psi(1)=0$. 

The solution exhibits a boundary layer of width $\epsilon$ at $x=0$ and of width $\sqrt{\epsilon}$ at $y=0$. For all mesh resolutions used, both layers are contained within a single element layer.
In all the computations that follow we consider $\epsilon = 10^{-6}$.

Before analyzing the error norms, we first discuss the qualitative behavior of the numerical approximations, by observing
Figure~\ref{F:NumericalExample}, which presents the graphs obtained as solution of \eqref{eq:NumEx} for different stabilization methods for $N=80$. 
As is well known, in convection-dominated regimes the main challenge is the correct resolution of the boundary layers, while avoiding spurious oscillations. Standard numerical solutions often exhibit oscillatory behavior near layers and may violate the expected smoothness or maximum-principle behavior.

\begin{figure}[h!tbp]
        \includegraphics[width=.35\textwidth]{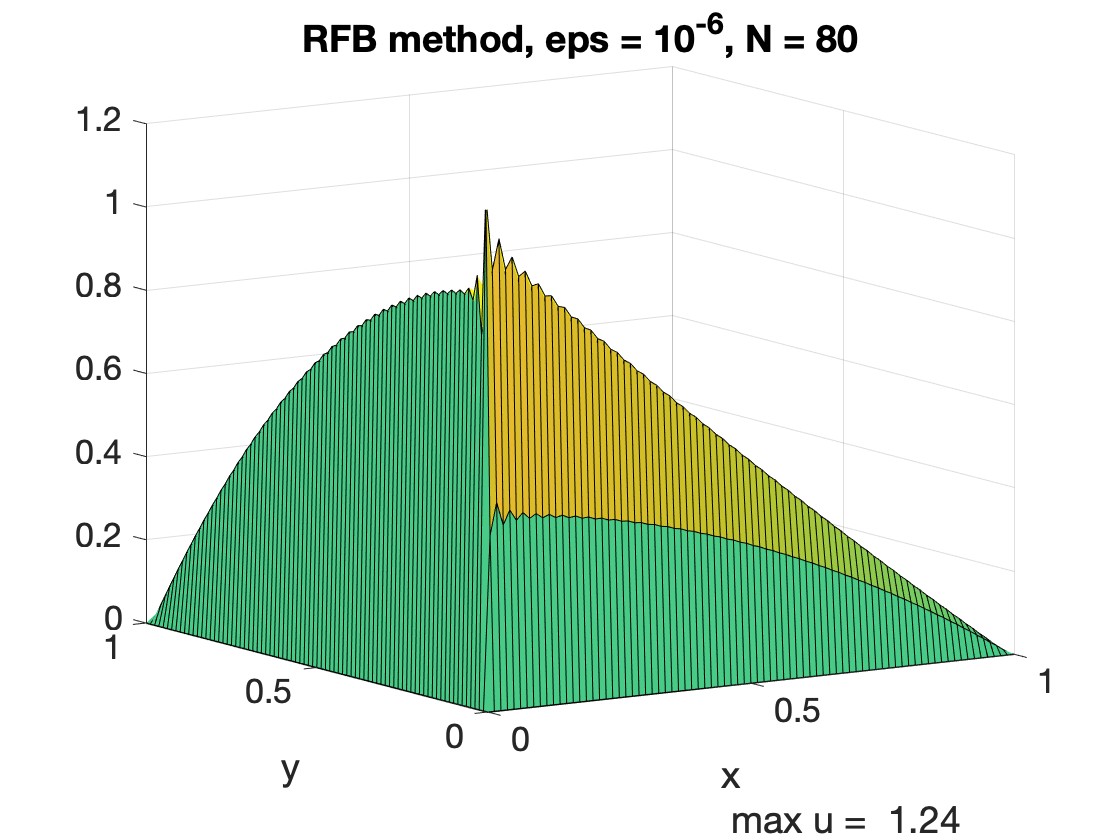}
        \hspace{-0.5 cm}
       \includegraphics[width=.35\textwidth]{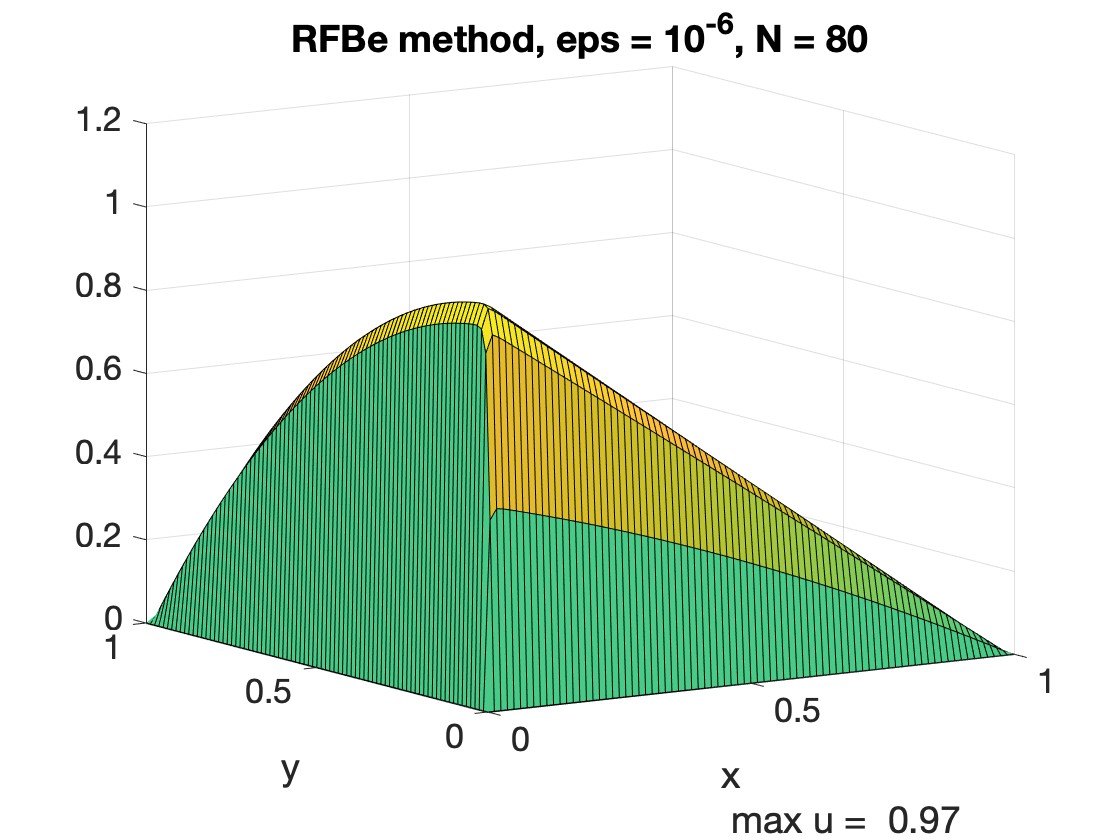}
       \hspace{-0.5 cm}
\includegraphics[width=.35\textwidth]{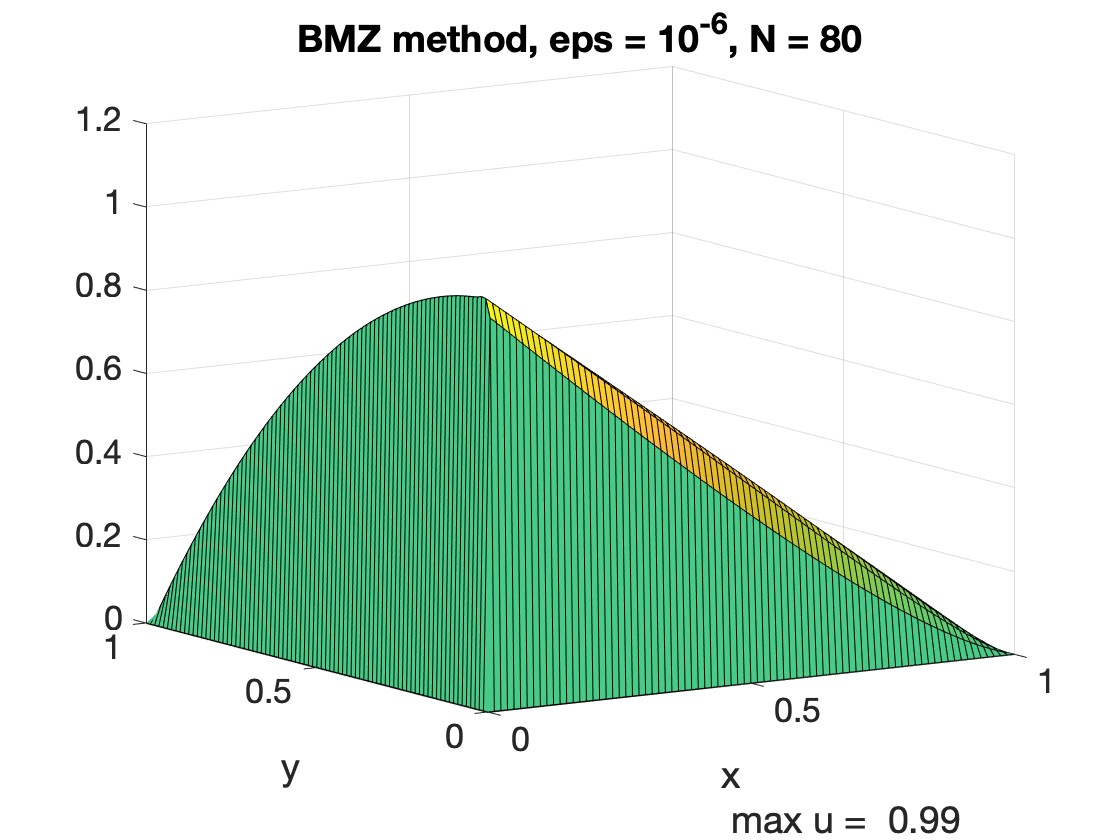}
    \caption{\small \label{F:NumericalExample} Example:
    Comparison of stabilization using RFB (left), RFBe (middle), and BMZ method (right). The solutions correspond to~\eqref{eq:NumEx}
with $\epsilon = 10^{-6}$, and $N =80$.}
\end{figure}

Near the bottom boundary $y=0$, the classical Residual-Free Bubble method (left) exhibits visible oscillations, reaching a peak value of approximately $1.2351$.
However,  ${\|u\|_{\infty}=1-\mathcal{O}(\epsilon)\approx 1}$, with its maximum attained near $(0,0).$
In contrast, both RFBe (middle) and BMZ (right) yield smooth, non-oscillatory approximations. The maximum attained by RFBe is approximately $0.9664$, while BMZ achieves $0.9890$, which is the closest to the expected value.

It is also instructive to inspect how the methods capture the two boundary layers.
Along the vertical layer at $x=0$, the BMZ solution decays to zero within a single layer of elements, as expected for a layer of width $\epsilon$. The RFBe approximation requires roughly two layers to achieve the same decay.
Along the bottom layer at $y=0$, the RFBe solution needs at least three elements before vanishing, whereas the BMZ method resolves the layer within only two elements.

In the analysis presented along this article the previous sections we used the norms $\|\cdot\|_{\epsilon}$ and $\vvvert \cdot \vvvert$ defined in~\eqref{eq:epsnorm} and~\eqref{eq:||| norm}.
For any $v\in \Hoi[\Omega],$
\begin{equation*}
\vvvert v \vvvert \ge 
\|v\|_{\epsilon}=
\left(\epsilon\|\nabla v\|_{L^2}^2
+ \|v\|_{L^2}^2\right)^{1/2},\quad\text{so that}\quad
\|v\|_{\epsilon}\ge
\| v\|_{L^2},\
\sqrt{\epsilon}\|\nabla v\|_{L^2}.
\end{equation*}
Tables~\ref{table: EOC L2 (2h,1-2h)} and ~\ref{table: EOC H1} reports the $L^2$ norm, and $H^1$–seminorm errors.
We additionally include the stability norm from~\cite{Brezzi1999}, in Table~\ref{table: EOC  stab}, offering a complementary way to evaluate the approximation, especially in terms of its derivative. It is defined by
\begin{equation*}\label{stab-norm}
    \left(
        \epsilon \|\nabla v\|_{L^2(\Omega)}^2 + 
        \sum_T h_T\|\aaa \cdot \nabla v\|_{L^2(T)}^2\right)^{1/2}.
\end{equation*}

When computing the norms, all degrees of freedom are taken into account.
According to~\eqref{eq:discrete problem}, the analyzed difference is given by:  
\begin{equation*}
    u - u_h = u - (u_L + u_B).
\end{equation*}    

Classically, only the nodal component $u - u_L$ is considered.
However, including the bubble contributions $u_B$ yields a more accurate measure of the approximation, since it incorporates information from the interior of each element.
To avoid the influence of the boundary layers, the norms are computed in the reduced domain $(2h,1-2h)^2$.

In Table \ref{table: EOC H1} we present the $H^1$-seminorm of the errors for the three methods.
It is worth noting that both RFBe and RFB exhibit nearly constant errors across all mesh refinements. In contrast, the BMZ method shows a clear decrease of the error as the mesh is refined: it reaches a value of $0.225$ for $N=160$, which is roughly half of the corresponding RFBe error. 

\begin{table}[ht]
\centering
\begin{tabular}{|c|c|c|c|c|c|c|}
\hline
$N$ & BMZ & EOC & RFBe & EOC & RFB & EOC \\
\hline
10  &  3.195\,\, &   -    & 0.428\,\,&   -   & 0.635\,\, &   -  \\
20  &  2.486\,\, &  0.3   & 0.422\,\,&  0.0  & 0.739\,\, & -0.2 \\
40  &  1.400\,\, &  0.8   & 0.492\,\,& -0.2  & 0.798\,\, & -0.1 \\
80  &  0.531\,\, &  1.3   & 0.613\,\,& -0.3  & 0.901\,\, & -0.1 \\
160 &  0.225\,\, &  1.2   & 0.578\,\,&  0.0  & 1.059\,\, & -0.2 \\
\hline
\end{tabular}
\caption{\small \label{table: EOC  H1} $H^1$ errors in $(2h,1-2h)^2$ and experimental order of convergence. We consider the solutions of \eqref{eq:NumEx} for different stabilization methods with ${\epsilon=10^{-6}}$. It is worth noting that both RFBe and RFB display almost mesh-independent errors across all refinements, whereas the BMZ approach exhibits a clear reduction of the error as the mesh is refined.}
\end{table}

Table \ref{table: EOC stab} reports the errors in the stability norm for the three methods.
Overall, the three approaches exhibit errors of comparable magnitude, although their convergence behaviors differ noticeably. For coarse meshes $(N=10)$, RFBe achieves the smallest error, approximately a factor of $2.5$ and $6$ smaller than RFB and BMZ, respectively.
Regarding the experimental order of convergence (EOC), RFBe displays a rate below $1$ across all refinements. In contrast, the BMZ method shows a markedly superior performance: for $20\le N\le 160$ its EOC ranges between $1.1$ and $2.0$, significantly outperforming both RFB and RFBe. 

\begin{table}[ht]
\centering
\begin{tabular}{|c|c|c|c|c|c|c|}
\hline
$N$ & BMZ & EOC & RFBe & EOC & RFB & EOC \\
\hline
10  & 18.211\,e\,-2 &    -    & 3.099\,e\,-2&   -  & 7.881\,e\,-2 &  -  \\
20  &  8.074\,e\,-2 &   1.1   & 1.807\,e\,-2&  0.7 & 6.968\,e\,-2 & 0.1 \\
40  &  2.851\,e\,-2 &   1.5   & 0.930\,e\,-2&  0.9 & 5.172\,e\,-2 & 0.4 \\
80  &  0.720\,e\,-2 &   1.9   & 0.510\,e\,-2&  0.8 & 3.393\,e\,-2 & 0.6 \\
160 &  0.169\,e\,-2 &   2.0   & 0.341\,e\,-2&  0.5 & 1.985\,e\,-2 & 0.7 \\
\hline
\end{tabular}
\caption{\small \label{table: EOC  stab} Stability norm errors in $(2h,1-2h)^2$ and experimental order of convergence. We consider the solutions of \eqref{eq:NumEx} for different stabilization methods with ${\epsilon=10^{-6}}$.
For $N=160$, the BMZ method yields the smallest error, approximately half of the RFBe error and about one order of magnitude smaller than that obtained with RFB. Regarding the experimental order of convergence, RFBe exhibits rates below $1$ for all mesh refinements, whereas BMZ shows a significantly better performance, with EOC values ranging between $1.1$ and $2.0$. }
\end{table}


 
 Table \ref{table: EOC L2 (2h,1-2h)} reports the $L^2$ errors for the three methods.
Although BMZ exhibits the largest error on the coarsest mesh ($N=10$), its convergence behavior is markedly superior: the experimental order of convergence for BMZ is consistently higher than that of RFBe and RFB at every refinement level shown. As a consequence, the BMZ error decreases more rapidly, eventually becoming the smallest among the three methods. For instance, at $N=160$, the BMZ error is already smaller by a factor than the RFBe error. It is important to note, however, that to the best of our knowledge, no theoretical error estimates have been developed for the $L^2$ norm for any of the methods considered in this article.

\begin{table}[ht]
\centering
\begin{tabular}{|c|c|c|c|c|c|c|}
\hline
$M$ & BMZ & EOC & RFBe & EOC & RFB & EOC \\
\hline
10  & 17.066\,e\,-3 &   -    & 3.838\,e\,-3&  -   & 9.566\,e\,-3 &  -  \\
20  &  7.307\,e\,-3 &  1.2   & 3.067\,e\,-3& 0.3  & 8.338\,e\,-3 & 0.1 \\
40  &  2.570\,e\,-3 &  1.5   & 2.330\,e\,-3& 0.3  & 6.024\,e\,-3 & 0.4 \\
80  &  0.578\,e\,-3 &  2.1   & 1.581\,e\,-3& 0.5  & 4.079\,e\,-3 & 0.5 \\
160 &  0.157\,e\,-3 &  1.8   & 0.813\,e\,-3& 0.9  & 2.683\,e\,-3 & 0.6 \\
\hline
\end{tabular}
\caption{\small \label{table: EOC L2 (2h,1-2h)}$L^2$ errors in $(2h,1-2h)^2$. We consider the solutions of \eqref{eq:NumEx} for different stabilization methods with ${\epsilon=10^{-6}}$. Although BMZ yields the largest error on the coarsest mesh, it converges significantly faster than RFBe and RFB. As a result, its error decreases more rapidly and becomes the smallest among the three methods under mesh refinement.}
\end{table}

In \cite{BMZ}, we extensively report on the excellent performance of the algorithm in different situations,
including cases with discontinuous boundary conditions, as well as examples involving non-parallel flows and interior layers.
We also extend our method to the time-dependent case.

\section{Conclusions}
\label{Sec:conclusions}

We have presented error estimates for the BMZ method~\cite{BMZ} applied to a convection dominated convection-diffusion equation, in the particular case of a channel flow in a square domain.
The error estimate is obtained in the energy norm
\[
\vvvert  v\vvvert   := \|v\|_\epsilon + 
\sup_{0\not= z_h\in V_h} \frac{a(v,z_h)}{\|z_h\|_\epsilon}.
\]
and states that, if $\ueps$ is the exact solution of the problem and $u_h$ the solution obtained with the BMZ method, then
$$ 
\vvvert \ueps - u_h\vvvert \le C (\epsilon/h)^{1/4},
$$ 
with a constant $C$ independent of $h$ and $\epsilon$, if $\epsilon/h < \hat{\epsilon}_0$ for some $0<\hat{\epsilon}_0<1$ small enough.


\section{Acknowledgments}

P.\ Morin was partially supported by the DFG Research Training Group 2339 (RTG 2339) \emph{Interfaces, Complex Structures, and Singular Limits}.

P.\ Morin and I.\ Zocola were partially supported by Agencia Nacional de Promoci\'on Cient\'ifica y Tecnol\'ogica through grant PICT-2020-SERIE A-03820, and by Universidad Nacional del Litoral through grants CAI+D-2020 50620190100136LI and CAI+D-2024 85520240100018LI.


\begin{appendix}
\section{Appendix: Results by Gie/Jung/Temam}\label{Appendix:Temam}

This appendix presents the results established in \cite{GieEtAl:13} that are used throughout this work. Some of the results coincide with the cited original statement, and others are simple adaptations, for which we provide a proof. In~\cite{GieEtAl:13}, also higher-order asymptotic approximations are discussed, but the zeroth-order approximation (i.e., $n = 0$) is sufficient for our purpose. 

\subsection{Results on the physical domain}
We consider the asymptotic expansion of $\ueps$ from~\cite{GieEtAl:13} corresponding to $n=0$:
\begin{equation*}
\uas := u^0 + \text{correctors}.
\end{equation*}

In what follows, we establish the expression of each function involved in this expansion, presenting both its constitutive equation and its global behavior. To this end, we consider a rectangular domain $R = (0,x_{\max}) \times (0,y_{\max})$, where $x_{\max},y_{\max}$ are taken to be either $1$ or $2$ throughout this work.

\medskip 
\textit{$u^0$: $0$-th order approximation.}
Let $u^0$ be the solution to
\begin{equation*}
    -  \partial_x u^0 =f \quad\text{in $R$,}\qquad
        u^0 = 0 \quad\text{at } x=x_{\max},
\end{equation*}
In the particular case that $f = f(y)$, we get that $u^0(x,y) = f(y)(x_{\max}-x).$

\medskip 
\textit{$\varphi$: Parabolic boundary layer correctors (at $y=0$, $y=y_{\max}$).}
In order to simplify notation, we only develop the expression for the bottom boundary layer, corresponding to $y=0$. The definition is analogous for $y=y_{\max}$.

Let $\delta(x)$ be a smooth function such that
\begin{equation*}
    \delta(x) = 0 \quad \text{for } 0\le x \le \frac12 x_{\max},\qquad
    \delta(x) = 1 \quad \text{for } x \ge \frac34 x_{\max}.
\end{equation*}
Then, two auxiliary functions are defined
\begin{equation*}
    \gamma(x)=f(x_{\max},0)(x-x_{\max})\delta(x),\qquad
    h(x)=-u^0(x,0) -\gamma(x).
\end{equation*}

In $R$, the corrector is established as the solution of
\begin{equation*}
    -\epsilon\,\partial_{y}^2\varphi-\partial_x\varphi = 0, \quad
    \varphi(x,0)=h(x), \quad 
    \varphi(x_{\max},y)=0.
\end{equation*}

\medskip 
\textit{$\xi$: Elliptic boundary layer correctors (at $x=x_{\max}$, and $y=0$, $y=y_{\max}$).}
In order to simplify notation, we only develop the expression for the bottom right corner, corresponding to $(x,y)=(x_{\max},0)$. The definition is analogous for $(x,y)=(x_{\max},y_{\max})$.
In $R$, the corrector is established as the solution of
\begin{equation*}
    -\epsilon\Delta \xi -\partial_x\xi = 0, \qquad
    \xi(x,0)=\gamma(x), \qquad 
    \xi(0,y)=0,\qquad 
    \xi(x_{\max},y)=0.
\end{equation*}

\medskip 
\textit{$\theta$: Ordinary boundary layer correctors (at $x=0$).}
In $R$, the corrector is established as the solution of
\begin{equation}\label{def:theta}
    -\epsilon\partial_x^2\theta = \partial_x\theta, \qquad
    \theta(0,y)=-u^0(0,y), \qquad 
    \lim_{x \to \infty} \theta(x,y) = 0.
\end{equation}
In the case that $f\equiv 1$, we get that $\theta(x,y) = -x_{\max}e^{-x/\epsilon}.$ 

\medskip 
\textit{$\zeta$: Ordinary corner layer correctors (at $x=0$, and $y=0$, $y=y_{\max}$).}
In order to simplify notation, we only develop the expression for the bottom left corner, corresponding to $(x,y)=(0,0)$. The definition is analogous for $(x,y)=(0,y_{\max})$.
In $R$, the corrector is established as the solution of
\begin{equation*}
    -\epsilon\partial_x^2\zeta = \partial_x\zeta, \qquad
    \zeta(0,0)=-\varphi(0,y), \qquad 
    \lim_{x \to \infty} \zeta(x,y) = 0.
\end{equation*}
from where it follows that $\zeta(x,y) = -\varphi(0,y)e^{-x/\epsilon}.$

\medskip 
\textit{$\eta$: Supplementary corner layer correctors (at ${x=0}$, and ${y=0}$, ${y=y_{\max}}$).}
In \cite{GieEtAl:13}, additional corner correctors are introduced. However, since we are considering the case $n=0$, these correctors $\eta$ vanish. In particular, the corrector corresponding to the bottom left corner is given by
\begin{equation*}
\begin{aligned}
\eta(x,y) &
= -(y_{\max} - y)\left(\theta(0, 0) + \zeta(0, 0)\right) 
= (y - y_{\max})
\left(-u^0(0,0) - \varphi(0, 0)\right) \\
&=(y - y_{\max})
\left(-u^0(0,0) - h(0)\right)
= (y - y_{\max})
\left(-u^0(0,0) +u^0(0,0)\right)\equiv 0.
\end{aligned}
\end{equation*}

We now state the results from~\cite{GieEtAl:13} adapted to our situation, corresponding to the asymptotic approximation which is obtained when considering $n=0$. In all the results that follow, the constant $\kappa$ depends on $x_{\max}$ and $y_{\max}$, and is otherwise independent of $\epsilon$ for $0<\epsilon<\epsilon_0$.

\begin{lemma}[Theorem 9,\cite{GieEtAl:13}]
\label{Temam:ErrorBound}
Let $\ueps$ be the exact solution and $\uas$ its asymptotic approximation corresponding to $n=0$. Then
\[
\|\ueps - \uas\|_{\epsilon}\leq \kappa \, \|f\| \, \epsilon.
\]
\end{lemma}

\begin{lemma}[Lemma 1,\cite{GieEtAl:13}]\label{Temam:Lemma1}
Let $\varphi$ be the corrector in the parabolic boundary layer. Then,
\begin{equation*}
|\varphi|\leq \kappa \, \|f\| \, \exp(-c \frac{y}{\sqrt{\epsilon}}),\quad \text{ and }\quad
|y\varphi|\leq \kappa \, \|f\| \, \epsilon^{1/2}\exp(-c \frac{y}{\sqrt{\epsilon}}).
\end{equation*}
\end{lemma}

\begin{lemma}[Lemma 2,\cite{GieEtAl:13}]\label{Temam:Lemma2}
Let $\varphi$ be the corrector in the parabolic boundary layer. Then
\[
\|\varphi\|\leq \kappa \, \|f\| \,  \epsilon^{1/4},
\quad \text{ and }\quad
\|\partial_x\varphi\| + \|\partial_y\varphi\| \leq  
\kappa \, \|f\| \, \epsilon^{-1/4}.
\]
\end{lemma}

\begin{lemma}[Lemma 3,\cite{GieEtAl:13}]\label{Temam:Lemma3}
Let $\theta$ be the corrector in the elliptic boundary layer. Then
\[
|\theta|\leq \kappa \, \|f\| \, \exp(-c\frac{x}{\epsilon}).
\]
\end{lemma}
\begin{lemma}[Lemma 5,\cite{GieEtAl:13}]\label{Temam:Lemma5}
Let $\zeta$ be the corrector in the corner layer. Then
\[
|\zeta|\leq \kappa \, \|f\| \, \exp(-c\left(\frac{x}{\epsilon}+\frac{y}{\sqrt{\epsilon}}\right)).
\]
\end{lemma}

\begin{lemma}[Lemma 6,\cite{GieEtAl:13}]\label{Temam:Lemma6}
Let $\zeta$ be the corrector in the corner layer. Then
\[
\|\zeta\|\leq \kappa \, \|f\| \, \epsilon^{3/4},
\quad \text{ and }\quad
\|\partial_x\zeta\| + \|\partial_y\zeta\| \leq  
\kappa \, \|f\| \,\epsilon^{-1/4}.
\]
\end{lemma}

\begin{lemma}[Eq. 35,\cite{GieEtAl:13}]\label{Temam:eq35}
Let $\xi$ be the corrector in the corner layer. Then
\[
\|\xi\|_\epsilon\leq \kappa \, \|f\| \, \epsilon^{1/4}.
\]
\end{lemma}

\subsection{Results for the scaled bubbles on the reference domains}

Let $\psi_\lambda$ be an element- or a patch-bubble with support $\Omega_\lambda$. Let $\hat{\psi}_\lambda$ be the scaled bubble to the reference domain $\hat \Omega_\lambda$,
\[
\hat \Omega_\lambda = 
\begin{cases}
(0,1)\times(0,1), \quad&\text{if $\Omega_\lambda$ is an element,} \\
(0,2)\times(0,1), \quad&\text{if $\Omega_\lambda$ is a patch corresponding to a vertical edge,}\\
(0,1)\times(0,2), \quad&\text{if $\Omega_\lambda$ is a patch corresponding to a horizontal edge.}
\end{cases}
\]
From~\eqref{eq:defHAtPsi},  $\hat{\psi}$ satisfies
\[
-\hat{\epsilon} \hat{\Delta} \hat{\psi} - \partial_{\hat{x}}\hat{\psi} = h \, \hat l(\hat y)\quad \text{in }\hat{\Omega}_\lambda,\qquad
   \hat{\psi} = 0\quad\text{on }\partial\hat{\Omega}_\lambda,
\]
with $\hat\epsilon = \epsilon/h$, and 
\[
\hat l(\hat y) = 
\begin{cases}
1-\hat y, \quad&\text{if $\lambda \in \Lambda_B$}, \\
\hat y, \quad&\text{if $\lambda \in \Lambda_T$}, \\
1, \quad&\text{otherwise}.
\end{cases}
\]

\begin{lemma}\label{Temam:BubbleCorrectorsBound}
Let $\psi_\lambda$ be a bubble, and let $\chi \in\{\varphi_\lambda,\xi_\lambda,\zeta_\lambda\}$ be one of its correctors. Then
\[
\|\hat{\chi}\|_{\hat{\epsilon},\hat{\Omega}_\lambda}\lesssim h \, \hat{\epsilon}^{1/4},
\quad \text{ and }\quad
\epsilon \|\nabla\chi\|_{{\Omega}_\lambda}^2 + \frac{1}{h}\|\chi\|_{{\Omega}_\lambda}^2 \lesssim h^3 \,(\epsilon/h)^{1/2}.
\]
\end{lemma}

\begin{proof}
We apply Lemmas \ref{Temam:Lemma2}, \ref{Temam:Lemma6}, and \ref{Temam:eq35}, observing that the right-hand side of the equation has a factor $h$:
\[
\|\hat{\chi}\|_{\hat{\epsilon},\hat{\Omega}_\lambda
}\leq \kappa h \hat{\epsilon}^{1/4}.
\]

Scaling back and using $\hat{\epsilon}=\epsilon/h$ yields the desired results.
\end{proof}

\begin{lemma}[App. B \cite{GieEtAl:13}, Thm. 3.6 \cite{ShihKellogg:87}]\label{Temam:Ap}
Let $\lambda \in \Lambda$ and let $\xi_\lambda$ be an elliptic corner corrector of the bubble $\psi_\lambda$. Then 
$$
|\hat{\xi}_\lambda(\hat{x},\hat{y})| \lesssim h\exp(-c(\sqrt{X^2 + Y^2} -X)),
$$
with $X=(\hat x_{\text{max}}-\hat{x})/2\hat{\epsilon}$, $Y=\hat{y}/2\hat{\epsilon}$ so that
\begin{equation}\label{Eq:Xi}
\|\hat{\xi}\|_{\hat{\Omega}_\lambda} \lesssim h\,\hat{\epsilon}.
\end{equation}
\end{lemma}
\begin{lemma}\label{Temam:zetaBubble}
Let $\lambda \in \Lambda$ and let $\zeta
_\lambda$ be an ordinary corner corrector of the bubble $\psi_\lambda$. Then
$$
\|\hat{\zeta}_\lambda e^x\|_{\hat{\Omega}_\lambda} \lesssim h \,\hat{\epsilon}^{3/4},
\quad \text{ and }\quad
\|\hat{y}\hat{\zeta}_\lambda e^x\|_{\Gammaleft} \lesssim h \,\hat{\epsilon}^{3/4},
$$
where \(\Gammaleft := \left\{ (\hat{x}, \hat{y}) \in \partial \hat{\Omega}_\lambda \,:\, \hat{x} = \hat x_{\min} \right\},
\)
with $\hat x_{\min}$ denoting the minimal $\hat{x}$-coordinate of $\hat{\Omega}_\lambda$.
\end{lemma}

\begin{proof}
 The first inequality is an immediate application of Lemma \ref{Temam:Lemma6}.
 
To prove the second inequality, observe that $\hat{\zeta}_\lambda(\hat x_{\min},\hat{y}) = -\hat{\varphi}_\lambda(\hat x_{\min},\hat{y})$. Then, one can use Lemma~\ref{Temam:Lemma1} and follow the same steps of the proof of Lemma~\ref{Temam:BubbleCorrectorsBound} 
 to bound
$$
|\hat{y}\hat{\zeta}_\lambda(\hat x_{\min},\hat{y})| \lesssim  h\,\hat{\epsilon}^{1/2}\, \exp(-c\frac{\hat{y}}{\sqrt{\hat{\epsilon}}}),
$$
and therefore, 
$$
\|\hat{y}\hat{\zeta}_\lambda e^x\|_{\Gammaleft} \lesssim  h\,\hat{\epsilon}^{3/4}.
$$
\end{proof}


\section{Appendix: Auxiliary results}
\label{Appendix:aux}

The main goal of this appendix is to prove Propositions~\ref{propo: bh_inf_bound} and~\ref{Propo:AngleSpace2}, which were instrumental in the proof of Lemma~\ref{Lemma:10}, which in turn is a key step to prove the main result of Theorem~\ref{Thm:main}. The results of all the intermediate lemmas seem natural taking into account the intuitive structure of the parabolic and elliptic boundary layers in the solutions to convection-dominated problems; but the precise proofs that we develop are, sadly, not so short.

Let ${v_h = q_h + b_h}\in V_h$ with $q_h \in V_L$ being piecewise bilinear and $b_h \in V_B$  a bubble function. 
Throughout this section we let $T\in \T$ be an element belonging to the top or bottom of the domain, i.e., $T\cap \{y=0\} \neq \emptyset$ or $T\cap \{y=1\} \neq \emptyset$. Without loss of generality we assume that, in $T$, $b_h= b_1 + \alpha b_2 + \beta b_T$, where $b_1$ and $b_2$ are the patch bubbles corresponding to the left and right side of $T$, respectively, and $b_T$ the element bubble; sketches can be seen in Figure \ref{F:bubbles}. More precisely, scaling to the reference elements, we have
\begin{align*}
-\hat{\epsilon} \hat{\Delta} \hat{b}_1 - \hat{b}_{1,x} &= h \,\hat l(\hat{y}) 
\quad\text{in }(-1,1)\times (0,1), & \hat{b}_1|_{\partial((-1,1)\times (0,1))}=0 ,
\\
-\hat{\epsilon} \hat{\Delta} \hat{b}_2 - \hat{b}_{2,x} &= h \,\hat l(\hat{y}) 
\quad\text{in }(0,2)\times (0,1), & \hat{b}_2|_{\partial((0,2)\times (0,1))}=0 ,
\\
-\hat{\epsilon} \hat{\Delta} \hat{b}_{\hat{T}} - \hat{b}_{\hat{T},x} &= h \,\hat l(\hat{y}) 
\quad\text{in }(0,1)\times (0,1), & \hatb_\hatT|_{\partial((0,1)\times (0,1))}=0 .
\end{align*}

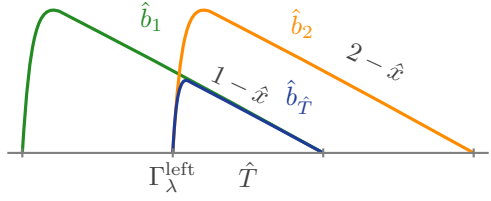
\begin{figure}[h]
    \centering
\begin{tikzpicture}[scale=2]
\node[gray!50!black!, below] at (0,0) {$\Gamma_{\lambda}^{\text{left}}$};
\node[gray!50!black!, below] at (0.5,0) {$\hat{T}$};

\draw[b1, very thick] plot [smooth] 
 (-1,0) .. controls (-0.95,0.8) and (-0.9,1).. (-0.75,0.94)
  .. controls (0,0.55)  .. (1,0);
\node[b1, above left] at (0,0.75) {$\hatb_1$};

\draw[b2, very thick] plot [smooth] 
(0,0) .. controls (0.05,0.8) and (0.1,1).. (0.25,0.94)
  .. controls (1,0.55)  .. (2,0);
\node[b2, above left] at (1,0.7) {$\hatb_2$};

\draw[bT, very thick] plot [smooth]
(0,0) .. controls (0.02,0.4) and (0.05,0.5).. (0.1,0.48)
  .. controls (0.5,0.27)  .. (1,0);
\node[bT, above left] at (1.0,0.2) {$\hatb_\hatT$};

\node[gray!50!black!,rotate=-25] at (0.45,0.45) {$1-\hatx$};
\node[gray!50!black!,rotate=-25] at (1.35,0.6) {$2-\hatx$};

\draw[gray, thick] (-1.1,0)--(2.1,0);

\draw[gray, thick] (-1,-0.03)--(-1,0.03);
\draw[gray, thick] (0, -0.03)--(0,0.03);
\draw[gray, thick] (1, -0.03)--(1,0.03);
\draw[gray, thick] (2, -0.03)--(2,0.03);
\end{tikzpicture}
\caption{\small \label{F:bubbles} Lateral view of the bubbles whose supports contain the element $T$.}
\end{figure}

In this appendix we establish a bound for
$\|\hatb_h/\haty\|_{\Gamma_\lambda^{\text{left}}}$ in Lemma~\ref{Lemma:bubble_bound_sup} and verify a strengthened Cauchy–Schwarz inequality involving $\|\hatN \hat{q}_h\|$, $\|\hatN \hatb_h\|$, and $\|\hatN \hat{v}_h\|$ in Lemma~\ref{lemma:Streng_CS}.
We begin by sketching the argument underlying the estimates, and then proceed with the detailed proofs. To prove the main estimate, established in Proposition~\ref{propo: bh_inf_bound}, we distinguish two cases.

We first consider the case in which
$\max\{|\oab|,|\otab|\}>\rho$, with $\rho>0$ small, in Lemma~\ref{Lemma:bubble_bound_inf_1}.
In this situation, even though the elliptic boundary layer along the left side ${\hatx=0}$ may be small due to cancellations, the boundary layer becomes large over a considerable portion of the bottom ${\haty=0}$ of $\hatT$. We restrict our attention to the bottom boundary layer, noting that the same argument applies, with minor notational changes, to the top boundary ${\haty=1}$.

To analyze this, we focus on the bottom rectangle
\(
R_B := (a,b)\times(0,\hatnsep),
\)
with $a=1/4$ and $b=3/4$ and $n\in\N$, such that $\hatnsep \le \frac12$, to be determined (See Figure \ref{F:R_BR_L}).

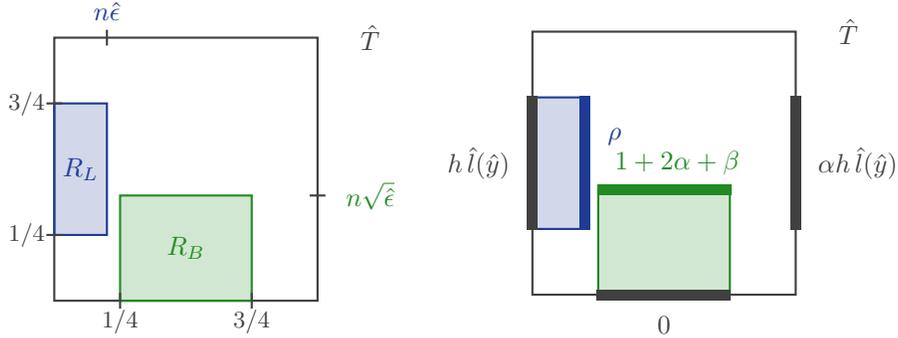
\begin{figure}[h]
    \centering
    \begin{tikzpicture}[scale=3.5]
        \draw[gray!50!black!, thick] (0,0) rectangle (1,1);

         \fill[bT, opacity=0.2] (0,0.25) rectangle (0.2,0.75);
         \draw[bT, thick] (0,0.25) rectangle (0.2,0.75);

         \fill[b1, opacity=0.2] (0.25,0.4) rectangle (0.75,0);
         \draw[b1, thick] (0.25,0.4) rectangle (0.75,0);
    
    \node[gray!50!black!] at (1.2,1) {$\hat{T}$};
    \node[b1] at (1.2,0.4) {$n\sqrt{\hat{\epsilon}}$};
    \node[bT] at (0.2,1.1) {$n\hat{\epsilon}$};
    \draw[gray!50!black!, thick] (0.2,0.97)--(0.2,1.03);
    \draw[gray!50!black!, thick] (0.97,0.4)--(1.03,0.4);
    \node[b1] at (0.5,0.2) {$R_B$};
    \node[bT] at (0.1,0.5) {$R_L$};

    \draw[gray!50!black!, thick] (0.25,-0.03)--(0.25,.03);
    \draw[gray!50!black!, thick] (0.75,-0.03)--(0.75,.03);
    \node[gray!50!black!,below] at (0.25,0) {\small $1/4$};
    \node[gray!50!black!,below] at (0.75,0) {\small $3/4$};

    \draw[gray!50!black!, thick] (-0.03,.25)--(.03,.25);
    \draw[gray!50!black!, thick] (-0.03,.75)--(.03,.75);
    \node[gray!50!black!,left] at (0,0.25) {\small $1/4$};
    \node[gray!50!black!,left] at (0,.75) {\small $3/4$};
    \end{tikzpicture}
    \quad
        \begin{tikzpicture}[scale=3.5]
        \draw[gray!50!black!, thick] (0,0) rectangle (1,1);

         \fill[bT, opacity=0.2] (0,0.25) rectangle (0.2,0.75);
         \draw[bT, thick] (0,0.25) rectangle (0.2,0.75);

         \fill[b1, opacity=0.2] (0.25,0.4) rectangle (0.75,0);
         \draw[b1, thick] (0.25,0.4) rectangle (0.75,0);
    
    \node[gray!50!black!] at (1.2,1) {$\hat{T}$};

     \fill[gray!50!black,] (0.98,0.245) rectangle (1.02,0.755);
     \node[gray!50!black!,right] at (1.05,0.5) {$\alpha h\,\hat l(\hat{y})$};
     \fill[gray!50!black,] (0.245,-0.02) rectangle (.754,0.02);
     \node[gray!50!black!,below] at (0.5,-0.05) {$0$};
     \fill[gray!50!black,] (-0.02,0.245) rectangle (0.02,0.755);
     \node[gray!50!black!,left] at (-0.05,0.5) {$h\,\hat l(\hat{y})$};
     
     \fill[b1] (0.245,0.38) rectangle (.755,0.42);
     \node[b1] at (0.55,0.5) {$\otab$};
     \fill[bT] (0.18,0.245) rectangle (.22,.755);
     \node[bT,right] at (0.25,0.6) {$\rho$};
    
    \end{tikzpicture}
    \caption{\small\label{F:R_BR_L}  Regions $R_L$ and $R_B$ (left). Approximate values on the element boundary and on the relevant edges: the top edge of $R_B$ and the right edge of $R_L$ (right).}
\end{figure}

We study the behavior of $\partial_\haty \hatb_h^{as}$ in $R_B$. 
On the right edge of $T$, $\hatb_h^{as}$ is approximately $\alpha\,h\,\hat l(\haty)$ (see Lemma~\ref{lemma:new_lemma}), and then increases toward the left with speed $\oab$. 

Therefore, at height $\haty=\hatnsep$ on the top of $R_B$, the function varies approximately from $\alpha$ to $\otab$. Since $\hatb_h^{as}$ vanishes at $\haty=0$, this implies that near the bottom boundary we have $|\partial_\haty \hatb_h^{as}|\cong1/\sqrt{\hat{\epsilon}}$.

Next, in Lemma~\ref{Lemma:bubble_bound_inf_2}, we consider the case
$\max\{|\oab|,|\otab|\}\le \rho$.
In this situation, although the parabolic boundary layer at the bottom ${\haty=0}$ may be small due to cancellations, a significant boundary layer develops along a substantial portion of the left side ${\hatx=0}$ of $\hatT$.

We therefore focus on the rectangle $R_L := (0,n\hat{\epsilon})\times(a,b)$ located on the left side of the domain, where the elliptic boundary layer appears. Here, $a=1/4$ and $b=1/2$ and $n\in\N$, such that $\hatnep \le \frac12$, to be determined (See Figure \ref{F:R_BR_L}).

When $\max\{|\oab|,|\otab|\}\le \rho$, the bubble $\hatb_h^{as}$ remains small near the left side at distance $\approx \hat{\epsilon}$, 
while at $\hatx=0$, we have $\hatb_h = \hatb_1\approx h \hat l(\haty)$ .
Therefore, $|\partial_\hatx \hatb_h^{as}|\cong 1/\hat{\epsilon}$ close to the left boundary of $\hatT$.


\begin{lemma}
\label{Lemma:bubble_bound_sup}
Let $\hat{b}_1$ be the patch bubble with $\supp(\hat{b}_1)=(-1,1)\times(0,1)$. Then:
\begin{equation*}
 \left\| \frac{\hat{b}_1(0,\haty)}{\haty} \right\|_{\{\hat{x}=0\}}^2 \lesssim
   \frac{1}{\sqrt{\hat{\epsilon}}}.
\end{equation*}
\end{lemma}

\begin{proof}
Use Hardy's inequality to bound 
\begin{equation*}
 \left\| \frac{\hat{b}_1(0,\haty)}{\haty} \right\|_{\{\hat{x}=0\}} \lesssim
   \left\| \partial_{\hat{y}} \hat{b}_1 \right\|_{\{\hat{x}=0\}}.
\end{equation*}

By \cite[Theorem 5.1]{KelloggStynes:05} $\partial_{\hat{y}} \hat{b}_1$ at
$\{\hat{x}=0\}$ is bounded by
$$
|\partial_{\hat{y}} \hat{b}_1(0,\haty)| \lesssim \hat{\epsilon}^{-1/2} 
    \exp(-c \haty/\sqrt{\hat{\epsilon}}),
$$
and from this 
\[
 \left\| \partial_{\hat{y}} \hat{b}_1 \right\|_{\{\hat{x}=0\}} \lesssim \hat{\epsilon}^{-1/4}.
\qedhere
\]
\end{proof}


\begin{lemma}\label{lemma:new_lemma}
Let $\hat{b}_2$ be the patch bubble with $\supp(\hat{b}_2)=[0,2]\times[0,1]$. Then there are constants $c,C>0$ independent of $\hat{\epsilon}$ such that, for all $n\leq \frac{1}{2\sqrt{\hat{\epsilon}}}$,
\begin{equation*}
| \hatb_2^{as}(1,\hat y) - h\,\hat l(\haty)|\lesssim \max\{A_n,B_n,C_n,D_n\}=\frac{h\exp(-cn)}{n}+h\exp(-cn^2) + \est 
\end{equation*}
for $\haty\in(\hatnsep,1-\hatnsep).$ Here,
\begin{equation*}
    A_n = \frac{h\exp(-cn)}{n},\hspace{0.2 cm}
    B_n = h\exp(\frac{-c n^2}{4}),
    \hspace{0.2 cm}
    C_n = h\exp(-\frac{c}{\hat{\epsilon}}),
    \hspace{0.2 cm}
    D_n = h\exp(-c(\frac{1}{\hat{\epsilon}}+n)).
\end{equation*}
\end{lemma}

\begin{proof}
For the proof, we are going to consider only the values for $\haty\in(\hatnsep,1/2),$ and consider the correctors at the bottom parabolic layer. 

Using the asymptotic expansion on \cite{GieEtAl:13} for the bubble $\hatb_2$ we can establish
\begin{equation*}
    |\hatb_2^{as}(1,\haty)-\hatb_2^0(1,\haty)|\leq |\varphi_b + \xi_b +\theta_b + \zeta_b| :=
    |\text{correctors of $\hatb_2$}|,
\end{equation*}
where the correctors  are defined following the definition contained in the Appendix~\ref{Appendix:Temam}.

Let $n\in\N$ with $n\leq \frac{1}{2\sqrt{\hat{\epsilon}}}$, and $\haty\in(\hatnsep,1/2)$, by  Lemma~\ref{Temam:Lemma1},  the first bound of Lemma~\ref{Temam:Ap}, Lemma~\ref{Temam:Lemma3} and Lemma~\ref{Temam:Lemma5} we get
\begin{align*}
    |\varphi_b(1,\haty)|\lesssim \frac{h}{n}\exp(-cn),
    &\quad
    |\xi_b(1,\haty)|\lesssim 
     h\exp(\frac{-c n^2}{4}),
     \quad\\
    |\theta_b(1,\haty)|\lesssim h\exp(-\frac{c}{\hat{\epsilon}}),
    &\quad
    |\zeta_b(1,\haty)|\lesssim h\exp(-c(\frac{1}{\hat{\epsilon}}+n)).
    \qedhere
\end{align*}




\end{proof}


\begin{lemma}
\label{Lemma:bubble_bound_inf_1}
Let $\hatb_h= \hatb_1 + \alpha \hatb_2 + \beta \hatb_T$ on $\hatT$ for some $\alpha,\beta\in \rz$, such that 
$\max\{|1+\alpha+\beta|,|1+2\alpha+\beta|\}>\rho$.
Then,
\begin{equation*}
\int_{\hatT}|\nabla \hatb_h^{as}|^2 
\ge
\int_{R_B}|\nabla \hatb_h^{as}|^2 
\gtrsim
    \frac{h^2(1+|\alpha|+|\beta|)^2}{\sqrt{\hat{\epsilon}}}.
\end{equation*} 

Here, $\hatb_h^{as}$ denotes the asymptotic expansion of $\hatb_h$, given by
\begin{equation}\label{bhas}
\hatb_h^{as} =  \hatb_1^{as} + \alpha \hatb_2^{as} + \beta \hatb_\hatT^{as},
\end{equation}
 where $\hatb_1^{as}$,  $\hatb_2^{as}$, $\hatb_\hatT^{as}$ are the asymptotic expansions given by~\cite{GieEtAl:13}, corresponding to  $\hatb_1$,  $\hatb_2$, $\hatb_\hatT$, respectively.
\end{lemma}
\begin{proof}

We define $\hatb_h^0 = \hatb_1^0 + \alpha \hatb_2^0 + \beta \hatb_\hatT^0$ with $\hatb_*^0$ the solution of the corresponding transport equation with homogeneous boundary value on the right side of the domain.
Then, for $\hatx\in(1/4,3/4)$,
\begin{equation}\label{eq:bh_b0}
\begin{split}
    |\hatb_h^{as}(\hatx,\hatnsep)-\hatb_h^0(\hatx,\hatnsep)|\leq{}& |\varphi_B(\hatx,\hatnsep)| + |\xi_B(\hatx,\hatnsep)|  
    \\
    &+|\theta_B(\hatx,\hatnsep)| + |\zeta_B(\hatx,\hatnsep)|.
\end{split}
\end{equation}
where $\chi_B = \chi_1 + \alpha\chi_2 + \beta\chi_T$, for $\chi \in \{\phi,\xi,\theta,\zeta\}$ are the correctors of each bubble.
In order to bound each of the terms on the right-hand side, we are going to follow the steps in the proof of Lemma~\ref{lemma:new_lemma} for $\hatx\in(a,b)$ and $\haty=\hatnsep$. By  Lemma~\ref{Temam:Lemma1},  Lemma~\ref{Temam:Lemma3}  and Lemma~\ref{Temam:Lemma5} we get
\begin{equation}
\label{eq:boundACD_2}
\begin{aligned}
    &|\varphi_B(\hatx,\hatnsep)| \lesssim h(1+|\alpha|+|\beta|)\frac{\exp(-cn)}{n},\\
    &|\theta_B(\hatx,\hatnsep)| \lesssim h(1+|\alpha|+|\beta|)\exp(-\frac{c}{4\hat{\epsilon}})=  h(1+|\alpha|+|\beta|)\est,\\
    &|\zeta_B(\hatx,\hatnsep)| 
    \lesssim h(1+|\alpha|+|\beta|)\exp(-c(\frac{1}{4\hat{\epsilon}}+n))
    =  h(1+|\alpha|+|\beta|)\est.
\end{aligned}
\end{equation}
Using the first bound of Lemma~\ref{Temam:Ap}, we obtain
\begin{equation}
\label{eq:boundB_2}
\begin{aligned}
    |\xi_B(\hatx,\hatnsep)|& \lesssim h(1+|\alpha|+|\beta|)
    \exp(-c\frac{\sqrt{(1-\hatx)^2+(\hatnsep)^2}-(1-\hatx)}{2\hat{\epsilon}})\\
    &  \lesssim h(1+|\alpha|+|\beta|)
    \exp(-c\frac{n^2\hat{\epsilon}}{\hat{\epsilon}}\frac{1}{1-\hatx})
    \\
    & 
    \lesssim h(1+|\alpha|+|\beta|)
    \exp(-\frac{c}{4}n^2).
\end{aligned}
\end{equation}

Putting together Eqs.~\eqref{eq:bh_b0},~\eqref{eq:boundACD_2},~\eqref{eq:boundB_2} we get
\begin{equation*}
\label{eq:bh0-right-2}
\begin{aligned}
        |\hatb_h^{as}(\hatx,\hatnsep)-\hatb_h^0(\hatx,\hatnsep)|
         &\lesssim 
         h(1+|\alpha|+|\beta|)\delta_n,
\end{aligned}
\end{equation*}
where $\delta_n \doteq e^{-cn}/n+e^{-cn^2}+\est$.
Therefore, for $\hatx\in(a,b)$, we obtain
\begin{equation}\label{eq:bound_bh2}
    |\hatb_h^0(\hatx,\hatnsep)|^2 - 
    C h^2(1+|\alpha|+|\beta|)^2\delta_n^2
    \leq 
    |b_h^{as}(x,\nsep)|^2.
\end{equation}

Using the Fundamental Theorem of Calculus in the $\haty$-direction, for $\hatx\in(a,b)$, and the Cauchy-Schwartz inequality, we get
\begin{align*}
    \left[\hatb_h^{as}(\hatx,\hatnsep)\right]^2 
    &= \left[\hatb_h^{as}(\hatx,\hatnsep)-\hatb_h^{as}(\hatx,0)\right]^2 
    = \left[\int_0^{\hatnsep}\partial_\haty \hatb_h^{as}(\hatx,\haty) d\haty\right]^2
    \\
    &\le 
     \hatnsep \int_0^{\hatnsep}|\partial_\haty \hatb_h^{as}(\hatx,\haty)|^2 d\haty.
 \end{align*}
Integrating over $\hatx\in(a,b)$ we obtain
\begin{equation*}
\frac{1}{\hatnsep}\int_a^b \hatb_h^{as}(\hatx,\hatnsep)^2\leq\int_{R_B}|\partial_\haty \hatb_h^{as}|^2.
\end{equation*}

Therefore, expanding the terms in~\eqref{eq:bound_bh2} we get
\begin{equation}\label{crucial bound}
\begin{split}
    \int_{R_B}|\hatN \hatb_h^{as}|^2 \ge {}&
    \int_{R_B} |\partial_\haty \hatb_h^{as}|^2
    \ge\frac{1}{\hatnsep}\int_a^b |\hatb_h^{as}(\hatx,\hatnsep)|^2d\hatx 
    \\
    \ge{}& \frac{1}{\hatnsep}
    \int_a^b \hatb_h^0(\hatx,\hatnsep)^2d\hatx
    -\frac{C_1 h^2}{\hatnsep}(1+|\alpha|+|\beta|)^2\delta_n^2 .
\end{split}
\end{equation}
Now, $\int_a^b \hatb_h^0(\hatx,\hatnsep)^2 \,d\hatx 
= l(n\sqrt{\hat{\epsilon}})^2 \int_a^b (\otab - (\oab)\hatx)^2\,d\hatx$, and 
\begin{align*}
I :={}&\int_a^b (\otab - (\oab)\hatx)^2\,d\hatx
\\
={}& (b-a)
\Big[(1+2\alpha+\beta)^2-(a+b)(1+2\alpha+\beta)(1+\alpha+\beta)\\
&+\frac{a^2 +ab+b^2}{3}(1+\alpha+\beta)^2\Big]    .
\end{align*}

\yesdelete{
Using that $a=1/4$, $b=3/4$, $a+b=1$, $b-a=1/2$, and $(a^2+ab+b^2)/3=13/48$ we have
\begin{align*}
I ={}& \frac12
    \Big[(1+2\alpha+\beta)^2-(1+2\alpha+\beta)(1+\alpha+\beta)
    +\frac{13}{48}(1+\alpha+\beta)^2\Big].
\end{align*}
Young's inequality, for $\eta>0$, yields
\begin{align*}
I \ge{}& \frac12\Big[\big(1-\frac{1}{4\eta}\big)(1+2\alpha+\beta)^2
+\big(\frac{13}{48}-\eta\big)(1+\alpha+\beta)^2\Big] 
\end{align*}
We now choose $1/4 < \eta < 13/48$ so that $\frac{13}{48}-\eta > 0$ and $1-\frac1{4\eta} >0$, whence
}

Replacing the values of $a$ and 
$b$, and applying Young’s inequality to the product $(\oab)(\otab)$, we obtain
\begin{equation}\label{Ibound}
I \gtrsim  (1+2\alpha+\beta)^2
+(1+\alpha+\beta)^2 .
\end{equation}

\yesdelete{
Now notice that,
\begin{equation*}
\frac{\alpha^2}2 \le (\oab)^2 +(\otab)^2,
\quad\text{ for all $\beta > 0$,}
\end{equation*}
due to the fact that, as a function of $\beta$,  $(\oab)^2+(\otab)^2$ attains its minimum at $\beta = -1-\frac{3\alpha}2$, and that minimum value is precisely $\alpha^2/2$. 
Also, 
\[
1 \lesssim \rho^2 \lesssim (1+2\alpha+\beta)^2
+(1+\alpha+\beta)^2 
\]
and 
\[
|\beta |^2 \le \big(|\oab|+1+|\alpha|\big)^2 \lesssim (1+2\alpha+\beta)^2
+(1+\alpha+\beta)^2 .
\]
}

Now, observing that
\begin{align*}
    &\frac{\alpha^2}2 \le (\oab)^2 +(\otab)^2,\\
&1 \lesssim \rho^2 \lesssim (1+2\alpha+\beta)^2
+(1+\alpha+\beta)^2,\\
&|\beta |^2 \le \big(|\oab|+1+|\alpha|\big)^2 \lesssim (1+2\alpha+\beta)^2
+(1+\alpha+\beta)^2. 
\end{align*}

We get
\[
(1+|\alpha|+|\beta|)^2
\lesssim (1+2\alpha+\beta)^2
+(1+\alpha+\beta)^2 
.
\]

Inserting this last bound in~\eqref{Ibound} and using that $h\, \hat l(\hatnsep)\ge 1/2$, we have that 
\[
\int_a^b \hatb_h^0(\hatx,\hatnsep)^2 \,d\hatx   \ge C_2h^2 (1+|\alpha|+|\beta|)^2,
\]
which using~\eqref{crucial bound} implies
\begin{align*}
    \int_{R_B}|\hatN b_h^{as}|^2 
    \ge{}& 
  \frac{h^2}{\hatnsep} \big( C_2 - \delta_n^2 C_1 \big)
  (1+|\alpha|+|\beta|)^2
\ge{} 
  \frac{Ch^2}{\hatnsep} 
  (1+|\alpha|+|\beta|)^2,
\end{align*}
if $n \in\N$ is large enough so that $\delta_n^2 \le C_2/C_1$.
\end{proof}

\begin{lemma}
\label{Lemma:bubble_bound_inf_2}
Let $\hatb_h= \hatb_1 + \alpha \hatb_2 + \beta \hatb_T$ on $\hatT$ for some $\alpha,\beta\in \rz$, such that 
$\max\{|\oab|,|\otab|\}\leq\rho$.
Then
\begin{equation*}
\int_\hatT|\hatN \hatb_h^{as}|^2 
\ge
\int_{R_L}|\hatN \hatb_h^{as}|^2 \gtrsim 
    \frac{ h^2(1+|\alpha|+|\beta|)^2}{\hat{\epsilon}}
    .
\end{equation*}  
\end{lemma}

\begin{proof}
Since
$|1+2\alpha+\beta|\leq\rho$ and $|1+\alpha+\beta|\leq\rho$, we have that
\begin{equation*}
    |\alpha|< 2\rho, \quad |1+\beta|< 3\rho,
    \quad\text{and thus}\quad  
    1+|\alpha|+|\beta| \lesssim 1 
    .
\end{equation*}

Analogous to the proof of Lemma~\ref{Lemma:bubble_bound_inf_1}, but considering derivative in the $\hatx$-direction, we obtain
\begin{equation}
\begin{aligned}\label{eq:int_RL}
    \int_{R_L}|\hatN \hatb_h^{as}|^2&
    \geq\int_{R_L}|\partial_\hatx \hatb_h^{as}|^2
    \geq \frac{1}{n\hat{\epsilon}}\int_a^b(|\hatb_h^{as}(n\hat{\epsilon},\haty) - \hatb_h^{as}(0,\haty)|)^2 d\haty. 
\end{aligned}
\end{equation}

Let us analyze the value of $\hatb_h^{as}$ on the left and right side of $R_L$. For $\hatx=0$, we consider the analysis in \cite[Eq. 59]{GieEtAl:13} for the value in the boundary of the support of each bubble. There it is established that the asymptotic approximation of $\hatb_2$ and $b_T$ vanish at $\hat{x}=0$. Therefore,
\begin{equation*}
    \hatb_h^{as}(0,\haty)=\hatb_1^{as}(0,\haty) + \alpha \hatb_2^{as}(0,\haty) +\beta \hatb_\hatT^{as}(0,\haty) = \hatb_1^{as}(0,\haty).
\end{equation*}
Then, using Lemma~\ref{lemma:new_lemma}, we get
\begin{equation*}
    |b_h^{as}(0,y) - hl(y)|\leq h\delta_n,
\end{equation*}
where, as before, $\delta_n := e^{-cn}/n+e^{-cn^2}+\est$, with $\delta_n^2 \le 1/4$. Finally, for $\haty \in (a,b),$ we obtain
\begin{equation}\label{eq:bh_0}
    |\hatb_h^{as}(0,\haty)| \geq \frac{1}{2} - h\delta_n.
\end{equation}

Now, let us bound the value of $\hatb_h^{as}$ for $\hatx=n\hat{\epsilon}.$ The analysis is similar to the one done in order to bound $\hatb(\hatx,\hatnsep)$ in Lemma~\ref{Lemma:bubble_bound_inf_1}.


The difference between the $\hatb_h^{as}$ and its $0$-th order approximation can be studied, analyzing each of its correctors:
\begin{equation}
\label{eq:bh_neps_correctors}    
    |\hatb_h^{as}(n\hat{\epsilon},\haty)-\hatb_h^0(n\hat{\epsilon},\haty)|\leq
    |\varphi_B(n\hat{\epsilon},\haty)| + |\xi_B(n\hat{\epsilon},\haty)| + |\theta_B(n\hat{\epsilon},\haty)| + |\zeta_B(n\hat{\epsilon},\haty)|.
\end{equation}
By Lemma~\ref{Temam:Lemma1}, Lemma~\ref{Temam:Lemma3}  and Lemma~\ref{Temam:Lemma5},  considering that $\haty\in[1/4,1/2]$, and $1+|\alpha|+|\beta|\lesssim 1$
\begin{equation}
\begin{aligned}
\label{eq:boundACD_3}
    &|\varphi_B(\hatnep,\haty)| \lesssim h(1+|\alpha|+|\beta|)\frac{\exp(-c)}{4\sqrt{\hat{\epsilon}}}
    \lesssim h\est,\\
    &|\theta_B(\hatnep,\haty)|\lesssim h(1+|\alpha|+|\beta|)\exp(-\frac{cn\hat{\epsilon}}{\hat{\epsilon}})\lesssim h\exp(-cn)
    \lesssim h\delta_n,
    \\
    &|\zeta_B(\hatnep,\haty)|\lesssim h(1+|\alpha|+|\beta|)\exp(-c(\frac{n\hat{\epsilon}}{\hat{\epsilon}}+\frac{1}{4\sqrt{\hat{\epsilon}}}))
    \lesssim h\est
\end{aligned}
\end{equation}
Using the first bound of Lemma~\ref{Temam:Ap}, we obtain
\begin{equation}
\label{eq:boundB_3}
\begin{aligned}
    |\xi_B(\hatnep,\haty)|&\lesssim h(1+|\alpha|+|\beta|)
    \exp(-c\frac{\sqrt{(1-\hatnep)^2+\haty^2}-(1-\hatnep)}{2\hat{\epsilon}})\\
    & \lesssim 
    h\exp(-c\frac{1-\hatnep}{2\hat{\epsilon}}\frac{\haty^2}{2(1-\hatnep)^2})
    \lesssim
    h\exp(-c\frac{1-\hatnep}{2\hat{\epsilon}}\frac{1}{4})
    \lesssim h\est
\end{aligned}
\end{equation}
Therefore, from Eqs. \eqref{eq:bh_neps_correctors}--\eqref{eq:boundB_3}, we get
\begin{equation}\label{eq:bhas-b0}
\begin{aligned}
|\hatb_h^{as}(n\hat{\epsilon},\haty)-\hatb_h^0(n\hat{\epsilon},\haty)|
\lesssim h\delta_n.
\end{aligned}
\end{equation}
Even more, because of the hypothesis, for $\haty\in(1/4,1/2)$, we get
\begin{equation}
\label{eq:bound b0}
|\hatb_h^0(n\hat{\epsilon},\haty)| = 
    |(\otab -(\oab)n\hat{\epsilon})h\,\hat l(\haty)|\leq h\frac32\rho .
\end{equation}
From Eqs.~\eqref{eq:bhas-b0} and \eqref{eq:bound b0}, we obtain
\begin{equation*}
\begin{aligned} 
|\hatb_h^{as}(\hatnep,\haty)|&
\leq |\hatb_h^{as}(\hatnep,\haty)-\hatb_h^0(n\hat{\epsilon},\haty)|
+ |\hatb_h^0(n\hat{\epsilon},\haty)|
\leq Ch\delta_n + \frac32h\rho
\le \tilde Ch \rho,
\end{aligned}
\end{equation*}
choosing $\delta_n \le \rho$.
Returning to~\eqref{eq:int_RL} we establish,
using~\eqref{eq:bh_0}
\begin{equation*}
    \begin{aligned}
    \int_{R_L}|\hatN \hatb_h^{as}|^2
    &\geq \frac{1}{n\hat{\epsilon}}\int_{1/4}^{1/2}|\hatb_h^{as}(n\hat{\epsilon},\haty) - \hatb_h^{as}(0,\haty)|^2 d\haty\\
    &\geq \frac{1}{n\hat{\epsilon}}\int_{1/4}^{1/2}|\hatb_h^{as}(0,\haty)|^2 - 4|\hatb_h^{as}(\hatnep,\haty)|^2 d\haty\\
    &\geq \frac{1}{n\hat{\epsilon}}\frac{1}{4}\Big((\frac{1}{2}-h\delta_n)^2 - 4\tilde{C}^2h^2\rho^2\Big)\gtrsim \frac{h^2}{\hat{\epsilon}}
    \gtrsim \frac{h^2(1+|\alpha|+|\beta|)^2}{\hat{\epsilon}}
    ,
    \end{aligned}
\end{equation*}
if $\rho$ is sufficiently small,
as we wanted to prove.
\end{proof}
\begin{propo}\label{prop:bubble_bound_inf}
Let $\hatb_h= \hatb_1 + \alpha \hatb_2 + \beta \hatb_T$ on $\hatT$ for some $\alpha,\beta\in \rz$.
Then,
\begin{equation*}
\int_\hatT|\hatN b_h^{as}|^2 \gtrsim
\frac{h^2(1+|\alpha|+|\beta|)^2}{\sqrt{\hat{\epsilon}}}.
\end{equation*}  
\end{propo}
\begin{proof}
The proof follows from Lemmas~\ref{Lemma:bubble_bound_inf_1} and~\ref{Lemma:bubble_bound_inf_2}. 
\end{proof}
\begin{propo}\label{propo: bh_inf_bound}
  Let $\hatb_h= \hatb_1 + \alpha \hatb_2 + \beta \hatb_T$ on $\hatT$ for some $\alpha,\beta\in \rz$. Then
  \begin{equation*}
  \frac{h^2(1+|\alpha|+|\beta|)^2}{\sqrt{\hat{\epsilon}}}
  \lesssim \|\hatN \hatb_h \|_\hatT^2 
  \quad\text{and}\quad
      \left\| \frac{\hat{b}_1(0,\haty)}{\haty} \right\|_{\{\hat{x}=0\}}^2 \lesssim
      \|\hatN \hatb_h \|_\hatT^2.
  \end{equation*}
\end{propo}

\begin{proof}
By the triangle inequality
\begin{equation*}
    \|\hatN \hatb_h \|_\hatT \ge \|\hatN \hatb_h^{as} \|_\hatT
    - \|\hatN (\hatb_h - \hatb_h^{as})\|_\hatT,
\end{equation*}
and by Theorem~\ref{Temam:ErrorBound}, since
$\hatb_h - \hatb_h^{as}  = \hatb_1-\hatb_1^{as} + \alpha(\hatb_2- \hatb_2^{as}) + \beta (\hatb_\hatT-\hatb_\hatT^{as})$ we have
\begin{equation*}
    \|\hatN (\hatb_h - \hatb_h^{as})\|
    \lesssim h(1+|\alpha|+|\beta|)\sqrt{\hat{\epsilon}}
   .
\end{equation*}
Therefore, using the estimates from Proposition~\ref{prop:bubble_bound_inf}, we obtain
\begin{equation*}
\label{eq:grad_bh_inf}
\begin{split}
    \|\hatN \hatb_h \|_\hatT &\ge \|\hatN \hatb_h^{as} \|_\hatT
    -C h(1+|\alpha|+|\beta|)\sqrt{\hat{\epsilon}}
    \\
    &\gtrsim
    h(1+|\alpha|+|\beta|) (
    \hat{\epsilon}^{-1/4} - \hat{\epsilon}^{1/2} )
   \gtrsim h(1+|\alpha|+|\beta|)\hat{\epsilon}^{-1/4}.
    \end{split}
\end{equation*}
This is the first assertion of the Proposition, the second one follows from this last estimate and Lemma~\ref{Lemma:bubble_bound_sup}.
\end{proof}

\begin{lemma}[Strengthened Cauchy-Schwarz inequality]
\label{lemma:Streng_CS}
Let ${v_h = q_h + b_h}\in V_h$ with $q_h \in V_L$  and $b_h \in V_B$. In $T\in \T$, holds
\begin{equation*}
    \int_\hatT\hatN \hat{q}_h \cdot \hatN \hat{b}_h \lesssim \hat{\epsilon}^{1/4}\|\hatN \hat{q}_h\| \|\hatN \hat{b}_h\|.
\end{equation*}
\end{lemma}

\begin{proof}

Since $\hat{q}_h$ is a bilinear function,  $\hatD \hat{q}_h = 0$.
 On the other hand, by definition,  we can decompose $\hatb_h = \hatb_1 + \alpha \hatb_2+\beta \hatb_T$. Since each bubble vanishes at the boundary of its support, we get
\begin{equation*}
\hatb_h(\hatx,0) = \hatb_h(\hatx,1) = 0.
\end{equation*}
Additionally, 
\begin{equation*}
    \hatb_h(0,\haty) = \hatb_1(0,\haty), \qquad
    \hatb_h(1,y) = \alpha \hatb_2(1,\haty).
\end{equation*}

By the maximum principle, we obtain that $0\leq \hatb_2(\hatx,\haty)\leq (2-\hatx)\,\hat l(\haty)$ and $0\leq \hatb_1(\hatx,\haty)\leq (1-\hatx)\, \hat l(\haty)$ in $[0,2]\times[0,1]$ and $[-1,1]\times[0,1]$, respectively.

Then,  integrating by parts, we get
\begin{align*}
    \int_\hatT \hatN \hat{q}_h \cdot\hatN \hatb_h & = 
    \int_\hatT - \hatD \hat{q}_h \, \hat{b}_h
    + \int_{\partial \hatT} \partial_\nu \hat{q}_h  \, \hat{b}_h
    \le
    \int_{\{\hatx=0\}} |\hatN \hat{q}_h| \, 1 + \int_{\{\hatx=1\}} |\hatN \hat{q}_h| \,|\alpha|
    \\
    &\lesssim \| \hatN  \hat{q}_h\| 
    (1+|\alpha|+|\beta|)
    \frac{\hat{\epsilon}^{1/4}}{\hat{\epsilon}^{1/4}}
    \lesssim \hat{\epsilon}^{1/4} \| \hatN \hat{q}_h \| \, \| \hatN \hat{b}_h\|.
\end{align*}
Where the last inequality holds by Proposition~\ref{propo: bh_inf_bound}.
\end{proof}

Then we finally get

\begin{propo}
\label{Propo:AngleSpace2}
There is $\hat{\epsilon}_0, 0 < \hat{\epsilon}_0 < 1$ such that for
${v_h = q_h + b_h}\in V_h$ with $q_h \in V_L$  and $b_h \in V_B$ it holds
\begin{equation*}
    \|\nabla q_h\|_T  + \|\nabla b_h\|_T \lesssim \|\nabla v_h\|_T ,
\end{equation*}
for all $T\in \T$ and $0<\epsilon / h <hat{\epsilon}_0$.
\end{propo}

 \begin{proof}
 With Lemma \ref{lemma:Streng_CS} we get
 \begin{align*}
 \|\nabla v_h\|_T^2 &= \| \nabla (q_h + b_h)\|^2 
 =
 \|\nabla q_h\|_T^2 +\|\nabla b_h\|_T^2 + 2 (\nabla q_h,\nabla b_h)_T
  \\
  &\gtrsim (1- \hat{\epsilon}^{1/2}) (\|\nabla q_h\|_T^2 +\|\nabla b_h\|_T^2 )
  \gtrsim \|\nabla q_h\|_T^2 +\|\nabla b_h\|_T^2,
 \end{align*}
if $\hat{\epsilon}$ is small enough.
 \end{proof}

\end{appendix}

\bibliographystyle{abbrv}
\bibliography{biblioRFB}

@article{BMZ,
      title={Patch bubbles for advection-dominated steady
      and unsteady problems}, 
      author={Eberhard Bänsch and Pedro Morin and Itatí Zocola},
      year={2025},
      eprint={2504.21835},
      journal={arXiv:2504.21835},
      primaryClass={math.NA},
      url={https://arxiv.org/abs/2504.21835}, 
}

@article{Brezzi1994,
  author  = {Brezzi, F. and Russo, A},
  title   = {Choosing bubbles for advection-diffusion problems},
  journal = {Math. Models Methods Appl.},
  volume  = {04},
  number  = {04},
  pages   = {571--587},
  year    = {1994},
  doi     = {https://doi.org/10.1142/S0218202594000327}
}

@article{Brezzi1999,
  author  = {Brezzi, F. and Hughes,  T. J. R. and  Marini, L. D. and  Russo, A. and S\"uli, E.},
  title   = {A priori error analysis of residual-free bubbles for advection-diffusion problems},
  journal = {SIAM J. Numer. Anal},
  volume  = {36},
  number  = {6},
  pages   = {1933--1948},
  year    = {1999}
}

@article{Cangiani2005long,
  author  = {Cangiani, A. and S\"uli, E.},
  title   = {Enhaced {RFB} method},
  journal = {Numerische Mathematik},
  volume  = {101},
  number  = {},
  pages   = {273--308},
  year    = {2005},
  doi     = {10.1007/s00211-005-0620-7}
}

@article{Cangiani2005short,
  author  = {Cangiani, A. and S\"uli, E.},
  title   = {Enhanced residual-free bubble method for convection–diffusion problems},
  journal = {International Journal for Numerical Methods in Fluids},
  volume  = {47},
  number  = {},
  pages   = {1307--1313},
  year    = {2005},
  doi     = {10.1002/fld.859}
}

@article{BFHR1997,
  author  = {Brezzi, F. and Franca, L.P. and Hughes, T. and Russo, A.},
  title = {$b=\int g$},
  year    = {1997},
  month   = {6},
  pages   = {329--339},
  volume  = {145},
  journal = {Computer Methods in Applied Mechanics and Engineering}
}

@article{Russo1997,
  author  = {Franca, L. and Russo, A.},
  year    = {1997},
  month   = {11},
  pages   = {},
  title   = {Approximation of the {S}tokes {P}roblem by 
{R}esidual-{F}ree {M}acro {B}ubbles},
  volume  = {4},
  journal = {East-West Journal of Numerical Mathematics}
}

@article{Franca1998,
  author  = {L.P. Franca and A. Nesliturk and M. Stynes},
  title   = {On the stability of residual-free bubbles for convection-diffusion problems and their approximation by a two-level finite element method},
  journal = {Computer Methods in Applied Mechanics and Engineering},
  volume  = {166},
  number  = {1},
  pages   = {35-49},
  year    = {1998},
  note    = {Advances in Stabilized Methods in Computational Mechanics},
  issn    = {0045-7825},
  doi     = {https://doi.org/10.1016/S0045-7825(98)00081-4}
}

@phdthesis{thesisSchieweck,
  author = {F. Schieweck},
  title  = {Eine asymptotisch angepa{\ss}te Finite-Element-Methode f\"ur singul\"ar gest\"orte elliptische Randwertaufgaben},
  school = {Wiss. Z. Techn. Universität Magdeburg},
  year   = 1987
}

@article {GieEtAl:13,
    AUTHOR = {Gie, Gung-Min and Jung, Chang-Yeol and Temam, Roger},
     TITLE = {Analysis of mixed elliptic and parabolic boundary layers with
              corners},
   JOURNAL = {Int. J. Differ. Equ.},
  FJOURNAL = {International Journal of Differential Equations},
      YEAR = {2013},
     PAGES = {Art. ID 532987, 13},
      ISSN = {1687-9643},
   MRCLASS = {35J25 (35B25 35C20)},
  MRNUMBER = {3055164},
       DOI = {10.1155/2013/532987},
       URL = {https://doi.org/10.1155/2013/532987}
}

@article {ShihKellogg:87,
    AUTHOR = {Shih, Shagi-Di and Kellogg, R. Bruce},
     TITLE = {Asymptotic analysis of a singular perturbation problem},
   JOURNAL = {SIAM J. Math. Anal.},
  FJOURNAL = {SIAM Journal on Mathematical Analysis},
    VOLUME = {18},
      YEAR = {1987},
    NUMBER = {5},
     PAGES = {1467--1511},
      ISSN = {0036-1410},
   MRCLASS = {35B25 (35C20 35J25 76W05)},
  MRNUMBER = {902346}
}

@article {KelloggStynes:05, 
    AUTHOR = {Kellogg, R. Bruce and Stynes, Martin},
     TITLE = {Corner singularities and boundary layers in a simple
              convection-diffusion problem},
   JOURNAL = {J. Differential Equations},
  FJOURNAL = {Journal of Differential Equations},
    VOLUME = {213},
      YEAR = {2005},
    NUMBER = {1},
     PAGES = {81--120},
      ISSN = {0022-0396},
   MRCLASS = {35J25 (35B25)},
  MRNUMBER = {2139339},
MRREVIEWER = {Xingbin Pan},
       DOI = {10.1016/j.jde.2005.02.011},
       URL = {https://doi.org/10.1016/j.jde.2005.02.011},
}
\end{document}